\documentclass[11pt]{amsart}
\usepackage[utf8]{inputenc}
\usepackage{enumerate,amssymb,xcolor,comment}
\usepackage[a4paper, footskip=1cm, headheight = 16pt, top=3.4cm, bottom=2.5cm,  right=2.5cm,  left=2.5cm ]{geometry}
\usepackage[colorlinks,linkcolor=blue,citecolor=red]{hyperref}
\usepackage[center]{caption}

\newtheorem{theorem}{Theorem}[section]
\newtheorem{lemma}[theorem]{Lemma}

\newtheorem{corollary}[theorem]{Corollary}

\usepackage[foot]{amsaddr}

\usepackage[T1]{fontenc}

\usepackage[leqno]{amsmath}

\makeatletter
\newcommand{\leqnomode}{\tagsleft@true}
\newcommand{\reqnomode}{\tagsleft@false}
\makeatother

\usepackage{latexsym}
\usepackage{amsfonts}

\usepackage{tikz}
\usepackage{float}

\def\dd{\hbox{-}}

\usepackage{marvosym}               

\DeclareMathOperator{\tw}{tw}
\DeclareMathOperator{\width}{width}

\newcounter{tbox}

\newcommand{\sta}[1]{\medskip\refstepcounter{tbox}\noindent{\parbox{\textwidth}{(\thetbox) \emph{#1}}}\vspace*{0.3cm}}

\newcommand{\mylongtitle}[1]{%
  \ifodd\value{page}%
    \protect\parbox{0.97\linewidth}{#1}\hfill%
  \else%
    \hfill\protect\parbox{0.97\linewidth}{#1}%
  \fi%
}

\usepackage{latexsym}
\usepackage{amsfonts}
\usepackage[leqno]{amsmath}
\usepackage{tikz}
\usepackage{float}
\usepackage{lmodern}
\usepackage[T1]{fontenc}

\def\dd{\hbox{-}}

\usepackage{marvosym}

\newcommand{\otherlabel}[2]{\protected@edef\@currentlabel{#2}\label{#1}}
\makeatother

\mathchardef\mh="2D

\title[Induced subgraphs and tree decompositions VIII.]{Induced subgraphs and tree decompositions\\
VIII. Excluding a forest in (theta, prism)-free graphs}

\author{Tara Abrishami$^{\ast \dagger}$}
\author{Bogdan Alecu$^{\ast \ast \mathparagraph}$}
\author{Maria Chudnovsky$^{\ast \amalg}$}
\author{Sepehr Hajebi $^{\mathsection}$}
\author{Sophie Spirkl$^{\mathsection \parallel}$}

\address{$^{\ast}$Princeton University, Princeton, NJ, USA}
\address{$^{**}$School of Computing, University of Leeds, Leeds, UK}
\address{$^{\mathsection}$Department of Combinatorics and Optimization, University of Waterloo, Waterloo, Ontario, Canada}
\address{$^{\dagger}$ Supported by NSF-EPSRC Grant DMS-2120644.}
\address{$^{\amalg}$ Supported by NSF-EPSRC Grant DMS-2120644 and by AFOSR grant FA9550-22-1-0083.} 
     \address{$^{\mathparagraph}$ Supported by DMS-EPSRC Grant EP/V002813/1.} 
\address{$^{\parallel}$ We acknowledge the support of the Natural Sciences and Engineering Research Council of Canada (NSERC), [funding reference number RGPIN-2020-03912].
Cette recherche a \'et\'e financ\'ee par le Conseil de recherches en sciences naturelles et en g\'enie du Canada (CRSNG), [num\'ero de r\'ef\'erence RGPIN-2020-03912]. This project was funded in part by the Government of Ontario.}
\date {\today}

\begin{document}
\maketitle
\begin{abstract}
Given a graph $H$, we prove that every (theta, prism)-free graph of sufficiently large treewidth contains either a large clique or an induced subgraph isomorphic to $H$, if and only if $H$ is 
a forest.
\end{abstract}

\section{Introduction}\label{intro}

All graphs in this paper are finite and simple unless specified otherwise. 
Let $G,H$ be graphs. We say that $G$ \emph{contains} $H$ if $G$ has an induced subgraph isomorphic to $H$, and we say $G$ is \emph{$H$-free} if $G$ does not contain $H$. For a family $\mathcal{H}$ of graphs we say $G$ is \textit{$\mathcal{H}$-free} if $G$ is $H$-free for every $H \in \mathcal{H}$. A class of graphs is \textit{hereditary} if it is closed under isomorphism and taking induced subgraphs, or equivalently, if it is the class of all $\mathcal{H}$-free graphs for some family $\mathcal{H}$ of graphs.

For a graph $G = (V(G),E(G))$, a \emph{tree decomposition} $(T, \chi)$ of $G$ consists of a tree $T$ and a map $\chi: V(T) \to 2^{V(G)}$ with the following properties: 
\begin{enumerate}[(i)]
\itemsep -.2em
    \item For every $v \in V(G)$, there exists $t \in V(T)$ such that $v \in \chi(t)$. 
    
    \item For every $v_1v_2 \in E(G)$, there exists $t \in V(T)$ such that $v_1, v_2 \in \chi(t)$.
    
    \item For every $v \in V(G)$, the subgraph of $T$ induced by $\{t \in V(T) \mid v \in \chi(t)\}$ is connected.
\end{enumerate}

For each $t\in V(T)$, we refer to $\chi(t)$ as a \textit{bag of} $(T, \chi)$.  The \emph{width} of a tree decomposition $(T, \chi)$, denoted by $\width(T, \chi)$, is $\max_{t \in V(T)} |\chi(t)|-1$. The \emph{treewidth} of $G$, denoted by $\tw(G)$, is the minimum width of a tree decomposition of $G$. 

Treewidth was first popularized by Robertson and Seymour in their graph minors project, and has attracted a great deal of interest over the past three decades. Particularly, graphs of bounded treewidth have been shown to be well-behaved from structural
\cite{RS-GMV} and algorithmic \cite{Bodlaender1988DynamicTreewidth} viewpoints.

This motivates investigating the structure of graphs with large treewidth, and especially, the substructures emerging in them. The canonical result in this realm is the Grid Theorem of Robertson and Seymour \cite{RS-GMV}, which describes the unavoidable subgraphs of graphs with large treewidth. For a positive integer  $t$, the {\em $(t \times t)$-wall}, denoted by $W_{t \times t}$, is a planar graph with maximum degree three and treewidth $t$ (see Figure~\ref{fig:5x5wall}; a formal definition can be found in \cite{wallpaper}).

\begin{theorem}[Robertson and Seymour \cite{RS-GMV}]\label{wallminor}
For every integer $t\geq 1$ there exists $w=w(t)\geq 1$ such that 
every graph of treewidth more than $w$
contains a subdivision of $W_{t \times t}$ as a subgraph.
\end{theorem}

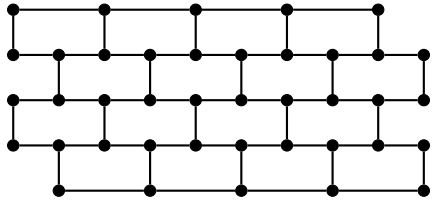
\begin{figure}
\centering

\begin{tikzpicture}[scale=2,auto=left]
\tikzstyle{every node}=[inner sep=1.5pt, fill=black,circle,draw]  
\centering

\node (s10) at (0,1.2) {};
\node(s12) at (0.6,1.2){};
\node(s14) at (1.2,1.2){};
\node(s16) at (1.8,1.2){};
\node(s18) at (2.4,1.2){};

\node (s20) at (0,0.9) {};
\node (s21) at (0.3,0.9) {};
\node(s22) at (0.6,0.9){};
\node (s23) at (0.9,0.9) {};
\node(s24) at (1.2,0.9){};
\node (s25) at (1.5,0.9) {};
\node(s26) at (1.8,0.9){};
\node (s27) at (2.1,0.9) {};
\node(s28) at (2.4,0.9){};
\node (s29) at (2.7,0.9) {};

\node (s30) at (0,0.6) {};
\node (s31) at (0.3,0.6) {};
\node(s32) at (0.6,0.6){};
\node (s33) at (0.9,0.6) {};
\node(s34) at (1.2,0.6){};
\node (s35) at (1.5,0.6) {};
\node(s36) at (1.8,0.6){};
\node (s37) at (2.1,0.6) {};
\node(s38) at (2.4,0.6){};
\node (s39) at (2.7,0.6) {};

\node (s40) at (0,0.3) {};
\node (s41) at (0.3,0.3) {};
\node(s42) at (0.6,0.3){};
\node (s43) at (0.9,0.3) {};
\node(s44) at (1.2,0.3){};
\node (s45) at (1.5,0.3) {};
\node(s46) at (1.8,0.3){};
\node (s47) at (2.1,0.3) {};
\node(s48) at (2.4,0.3) {};
\node (s49) at (2.7,0.3) {};

\node (s51) at (0.3,0.0) {};
\node (s53) at (0.9,0.0) {};
\node (s55) at (1.5,0.0) {};
\node (s57) at (2.1,0.0) {};
\node (s59) at (2.7,0.0) {};

\foreach \from/\to in {s10/s12, s12/s14,s14/s16,s16/s18}
\draw [thick] (\from) -- (\to);

\foreach \from/\to in {s20/s21, s21/s22, s22/s23, s23/s24, s24/s25, s25/s26,s26/s27,s27/s28,s28/s29}
\draw [thick] (\from) -- (\to);

\foreach \from/\to in {s30/s31, s31/s32, s32/s33, s33/s34, s34/s35, s35/s36,s36/s37,s37/s38,s38/s39}
\draw [thick] (\from) -- (\to);

\foreach \from/\to in {s40/s41, s41/s42, s42/s43, s43/s44, s44/s45, s45/s46,s46/s47,s47/s48,s48/s49}
\draw [thick] (\from) -- (\to);

\foreach \from/\to in {s51/s53, s53/s55,s55/s57,s57/s59}
\draw [thick] (\from) -- (\to);

\foreach \from/\to in {s10/s20, s30/s40}
\draw [thick] (\from) -- (\to);

\foreach \from/\to in {s21/s31,s41/s51}
\draw [thick] (\from) -- (\to);

\foreach \from/\to in {s12/s22, s32/s42}
\draw [thick] (\from) -- (\to);

\foreach \from/\to in {s23/s33,s43/s53}
\draw [thick] (\from) -- (\to);

\foreach \from/\to in {s14/s24, s34/s44}
\draw [thick] (\from) -- (\to);

\foreach \from/\to in {s25/s35,s45/s55}
\draw [thick] (\from) -- (\to);

\foreach \from/\to in {s16/s26,s36/s46}
\draw [thick] (\from) -- (\to);

\foreach \from/\to in {s27/s37,s47/s57}
\draw [thick] (\from) -- (\to);

\foreach \from/\to in {s18/s28,s38/s48}
\draw [thick] (\from) -- (\to);

\foreach \from/\to in {s29/s39,s49/s59}
\draw [thick] (\from) -- (\to);

\end{tikzpicture}

\caption{$W_{5 \times 5}$}
\label{fig:5x5wall}
\end{figure}

Theorem~\ref{wallminor} can also be reformulated into a full characterization of unavoidable minors in graphs of large treewidth: for every planar graph $H$, every graph of sufficiently large treewidth contains $H$ as a minor. In contrast, unavoidable induced subgraphs of graphs with large treewidth are far from completely understood. There are some natural candidates though, which we refer to as the ``basic obstructions'': complete graphs and complete bipartite graphs, subdivided walls mentioned above, and line graphs of subdivided walls, where the {\em line graph} $L(F)$ of a graph $F$ is the graph with vertex set $E(F)$, such that two vertices of $L(F)$ are adjacent if and only if  the corresponding edges of $F$ share an end. Note that the complete graph $K_{t+1}$, the complete bipartite graph $K_{t,t}$, and the line graph of every subdivision of $W_{t\times t}$ all have treewidth $t$. For a positive integer $t$, let us say a graph 
$H$ is a {\em $t$-basic obstruction} if $H$ is one of the following graphs:
$K_t$, $K_{t,t}$, a subdivision of $W_{t \times t}$, or  the line graph of
a subdivision of $W_{t \times t}$. We say a graph $G$ is \textit{$t$-clean} if $G$ does not contain a $t$-basic obstruction.

The basic obstructions do not form a comprehensive list of induced subgraph obstructions to bounded treewidth: there are $t$-clean graphs of arbitrarily large treewidth for small values of $t$. A well-known hereditary graph class demenstrating this fact is the class of even-hole-free graphs, where a {\em hole} is an induced cycle on at least four vertices, the \textit{length} of a hole is its number of edges and an \textit{even hole} is a hole with even length. In fact, complete graphs are the only even-hole-free basic obstruction. In other words, for every positive integer $t\geq 1$, one may observe that an even-hole-free graph is $t$-clean if and only if it is $K_t$-free. It is therefore tempting to ask whether even-hole-free graphs excluding a fixed complete graph have bounded treewidth. Sintiari and Trotignon \cite{layered-wheels} answered this with a vehement ``no'', providing a construction of (even-hole, $K_4$)-free graphs with arbitrarily large treewidth, hence proving that there are $t$-clean (even-hole-free) graphs of arbitrarily large treewidth for every fixed $t\geq 4$. In addition, graphs from this construction are sparse, in the sense that they exclude short holes.
\begin{theorem}[Sintiari and Trotignon \cite{layered-wheels}]\label{layeredwheel1}
    For all integers $w,l\geq 1$, there exists an (even-hole, $K_4$)-free graph $G_{w,l}$ of treewidth more than $w$ and with no hole of length at most $l$.
\end{theorem}

At the same time, $3$-clean even-hole-free graphs have treewidth at most five \cite{evenholetrianglefree}. So one may wonder whether all $3$-clean graphs have bounded treewidth. However, another construction by Sintiari and Trotignon \cite{layered-wheels} shows that in a superclass of even-hole-free graphs, namely the class of theta-free graphs, there $3$-clean graphs with arbitrarily large treewidth (see the next section for the definition of a theta; one may check that the every $t$-basic obstruction for $t\geq 3$ contains either a theta or a triangle). Indeed, the treewidth of theta-free graphs remains unbounded even when forbidding short cycles.

\begin{theorem}[Sintiari and Trotignon \cite{layered-wheels}]\label{layeredwheel2}
    For all integers $w,g\geq 1$, there exists a theta-free graph $G_{w,g}$ of treewidth more than $w$ and girth more than $g$.
\end{theorem}

A natural question to ask, then, is what further conditions must be imposed to guarantee bounded treewidth in even-hole-free graphs. For instance, graphs from both Theorems~\ref{layeredwheel1} and \ref{layeredwheel2} have vertices of arbitrarily large degree, and so it was conjectured in \cite{aboulker} that (theta, triangle)-free graphs of bounded maximum degree have bounded treewidth and that even-hole-free graphs of bounded maximum degree have bounded treewidth. These were proved in  \cite{wallpaper} and \cite{TWI}, respectively. In the same paper
\cite{aboulker}, a stronger conjecture was made, asserting that basic obstructions are in fact the only obstructions to bounded treewidth in graphs of bounded maximum degree. This was later proved in  \cite{Korhonen}, which closed the line of inquiry into graph classes of bounded maximum degree.

\begin{theorem}[Korhonen \cite{Korhonen}] \label{boundeddegree}
For all integers $t, \delta \geq 1$, there exists $w = w(t, \delta)$ such that every $t$-clean graph of maximum degree at most $\delta$ has treewidth at most $w$.
\end{theorem}

Despite its generality, the proof of Theorem~\ref{boundeddegree} is surprisingly short. However, the case of proper hereditary classes containing graphs of unbounded maximum degree seems to be much harder. For graph classes $\mathcal{G}$ and $\mathcal{H}$, let us say $\mathcal{H}$ \textit{modulates} $\mathcal{G}$ if for every positive integer $t$, there exists a positive integer $w(t)$ (depending on $\mathcal{G}$ and $\mathcal{H}$) such that every $t$-clean $\mathcal{H}$-free graph in $\mathcal{G}$ has treewidth at most $w(t)$. An induced subgraph analogue to Theorem~\ref{wallminor} is therefore equivalent to a full characterization of graph classes $\mathcal{H}$ which modulate the class of all graphs. This remains out of reach, but the special case where $|\mathcal{H}|=1$ turns out to be more approachable. For a graph $H$ and a graph class $\mathcal{G}$, let us say $H$ \textit{modulates} $\mathcal{G}$ if $\{H\}$ modulates $\mathcal{G}$.  Building on a method from \cite{lozin}, recently we characterized all graphs $H$ which modulate the class of all graphs:
\begin{theorem}[Abrishami, Alecu, Chudnovsky, Hajebi and Spirkl \cite{twvii}]\label{tw7}
    Let $H$ be a graph. Then $H$ modulates the class of all graphs if and only if $H$ is a subdivided star forest, that is, a forest in which every component has at most one vertex of degree more than two.
\end{theorem}
In general, for a hereditary class $\mathcal{G}$ containing $t$-clean graphs of arbitrarily large treewidth for small $t$, one can ask for a characterization of graphs $H$ modulating $\mathcal{G}$. Given Theorem~\ref{layeredwheel1}, a natural class $\mathcal{G}$ to consider is the class of even-hole-free graphs. Note that Theorem~\ref{layeredwheel1} shows that a graph $H$ modulates even-hole-free graphs only if $H$ is a chordal graph (that is, a graph with no hole) of clique number at most three. As far as we know, the converse may also be true, that every chordal graph of clique number at most three modulates even-hole-free graphs. In fact, in this paper we narrow the gap, showing that every chordal graph of clique number at most two, that is, every forest, modulates the class of even-hole-free graphs.
\begin{theorem}\label{mainevenhole}
    For every forest $H$ and every integer $t\geq 1$, every even-hole-free graph of sufficiently large treewidth contains either $H$ or a clique of cardinality $t$.
\end{theorem}

This aligns with the observation \cite{Trotignon} that every forest is contained in some graph $G_{w,l}$ from Theorem~\ref{layeredwheel1}. As mentioned above, one way to improve on Theorem~\ref{mainevenhole} is to push $H$ towards being an arbitrary chordal graph of clique number three. Another way to strengthen Theorem~\ref{mainevenhole} is to find a superclass $\mathcal{G}$ of even-hole-free graphs for which forests are the only graphs modulating $\mathcal{G}$. While the former remains open, we provide an appealing answer to the latter: our main result shows that forests are exactly the graphs which modulate the class of (theta, prism)-free graphs (see the next section for the definition of a prism; again one may check that in (theta, prism)-free graphs, being $t$-clean is equivalent to being $K_t$-free for every positive integer $t$). 

\begin{theorem}\label{mainfull}
    Let $H$ be a graph. Then $H$ modulates (theta, prism)-free graphs if and only if $H$ is a forest. In other words, given a graph $H$, for every integer $t\geq 1$, every (theta, prism)-free graph of sufficiently large treewidth contains either $H$ or a clique of cardinality $t$, if and only if $H$ is a forest.
\end{theorem}
Let $\mathcal{C}$ be the class of all  (theta, prism)-free graphs. It is easily seen that $\mathcal{C}$ contains all even-hole-free graphs, and so Theorem~\ref{mainfull} implies Theorem~\ref{mainevenhole}. Note that the ``only if'' direction of Theorem~\ref{mainfull} follows immediately from Theorem~\ref{layeredwheel2} as prisms contain triangles. Since every forest is an induced subgraph of a tree, in order to prove Theorem~\ref{mainfull}, it suffices to prove Theorem~\ref{mainthm} below, which we do in Section~\ref{getatree}. For a positive integer $t$ and a tree $F$, we denote by $\mathcal{C}_t$ the class of all graphs in
$\mathcal{C}$ with no clique of cardinality $t$ (that is, $t$-clean graph in $\mathcal{C}$), and by $\mathcal{C}_t(F)$ the class of all $F$-free graphs in $\mathcal{C}_t$.

\begin{theorem}\label{mainthm}
  For every tree $F$ and every integer $t \geq 1$, there exists an integer $\tau(F,t)\geq 1$ such that
  every graph in $\mathcal{C}_t(F)$ has treewidth at most $\tau(F,t)$.
\end{theorem}

We conclude this introduction by sketching our proofs (the terms we use here are defined in
later sections). The proof of Theorem~\ref{mainthm} begins with a two-step preparation. As the first step, inspired by a result from \cite{bisimp2}, we show that for every graph $G\in \mathcal{C}$ which contains a pyramid with certain conditions on the apex and its neighbors, $G$ admits a construction which we call a ``$(T,a)$-strip-structure,'' where $a$ is the apex of the pyramid and $T$ is an optimally chosen tree.
Roughly speaking, we show that $G\setminus \{a\}$ can be partitioned into two induced subgraphs $H$ and $J$ where $H$ is more or less similar to the line graph of the tree $T$ and every vertex in $J$ with a neighbor in $H$ attaches at a pyramid lurking in $H$ in a restricted way; we call the latter vertices ``jewels.'' The proof of this theorem occupies Sections~\ref{sec:pyramids} and \ref{sec:strips}. The second step is to employ the previous result to show that if $G \in \mathcal{C}_t$ admits a $(C,a)$-strip-structure where $C$ is a caterpillar, then every vertex in $G \setminus N_G[a]$ can be separated from $a$ by removing a few vertices (our proof works more generally when $C$ is any tree of bounded maximum degree, but the caterpillar case suffices for our application). We prove this in Section~\ref{sec:stripconnectivity}. The central difficulty in the proof is to deal with the jewels separately. This is surmounted in Section~\ref{sec:jewelsconnectivity} where we prove several results concerning the properties of jewels. Most notably, we show that jewels only attach at ``local areas of the line-graph-like part'' of $G$, and that only a few jewels attach at each local area. This concludes the preparation for proving  Theorem~\ref{mainthm}.

Next, we embark on the proof of Theorem~\ref{mainthm}. We assume that $G \in \mathcal{C}_t$ has large treewidth, which together with results from
Section~\ref{defns} implies that $G$ contains two vertices $x,y$ joined by
many pairwise internally disjoint induced paths $P_1, \dots, P_m$. Now we analyze the structure of the graph $G[P_1 \cup \dots \cup P_m]$. It turns out that, if $m$ is large enough, then either
\begin{itemize}
\item there are many paths among $P_i$'s whose union $H$ admits a $(C,x)$-strip-structure for some caterpillar $C$, or 
  \item for some large value of $d$, $G[P_1 \cup \dots \cup P_m]$ contains a tree $S$ isomorphic to the complete bipartite graph $K_{1,d}$, such that $x$ is the vertex of degree $d$ in $S$, and for every leaf $l$ of $S$, there are many pairwise internally disjoint induced paths between $l$ and $y$, such that in addition, paths corresponding to distinct leaves of $S$ are also pairwise internally disjoint.
\end{itemize}

The former case implies that $y$ can be separated from $x$ by removing few vertices, which using a result from Section~\ref{sec:stripconnectivity}, yields a contradiction with Menger's theorem.
The latter case is the first step towards building the large tree in $G$ as a subgraph. We now iterate the argument we just described, applying it to each leaf $l$ of $S$ and $y$, obtaining larger and larger trees. The process is stopped once we reach a sufficiently large tree as a subgraph of $G$. This, combined with the fact that $G\in \mathcal{C}_t$ and a result of Kierstead and Penrice \cite{KP}, yields the desired tree $F$ as an induced subgraph of $G$.

This paper is organized as follows. Section~\ref{defns} covers preliminary definitions as well as some results from the literature used in our proofs. Section~\ref{sec:pyramids}
investigates the behavior of pyramids in graphs from $\mathcal{C}$. Section~\ref{sec:strips} is devoted to defining strip-structures and jewels, and
showing how they arise from pyramids in graphs in $\mathcal{C}$. Section~\ref{sec:jewelsconnectivity} takes a closer look at jewels for the strip-structures obtained in Section~\ref{sec:strips}. In Section~\ref{sec:stripconnectivity} we show that admitting certain strip-structures weakens the connectivity of most vertices to the apex. Finally, in Section~\ref{getatree}, we prove
Theorem~\ref{mainthm}.

\section{Preliminaries and results from the literature}
\label{defns}
Let $G = (V(G),E(G))$ be a graph. For a set $X \subseteq V(G)$ we denote by $G[X]$ the subgraph of $G$ induced by $X$. For $X \subseteq V(G)\cup E(G)$, $G \setminus X$ denotes the subgraph of $G$ obtained by removing $X$. Note that if $X\subseteq V(G)$, then $G \setminus X$ denotes the subgraph of $G$ induced by $V(G)\setminus X$.  In this paper, we use induced subgraphs and their vertex sets interchangeably.

Let $x\in G$ and let $d$ be a positive integer. We denote by $N^d_G(x)$ the set of all vertices in $G$ at distance $d$ from some $x$, and by $N^d_G[x]$ the set of all vertices in $G$ at distance at most $d$ from $x$. We write $N_G(x)$ for $N^1_G(x)$ and $N_G[x]$ for $N^1_G[x]$. For an induced subgraph $H$ of $G$, we define $N_H(x)=N_G(x) \cap H$, $N_H[x]=N_G[x]\cap H$. Also, for $X\subseteq G$, we denote by $N_G(X)$ the set of all vertices in $G\setminus X$ with at least one neighbor in $X$, and define $N_G[X]=N_G(X)\cup X$. 

 Let $X,Y \subseteq G$ be disjoint. We say $X$ is \textit{complete} to $Y$ if all edges with an end in $X$ and an end in $Y$ are present in $G$, and $X$ is \emph{anticomplete}
to $Y$ if no edges between $X$ and $Y$ are present in $G$.

A {\em path in $G$} is an induced subgraph of $G$ that is a path. If $P$ is a path in $G$, we write $P = p_1 \dd \cdots \dd p_k$ to mean that $V(P) = \{p_1, \dots, p_k\}$ and $p_i$ is adjacent to $p_j$ if and only if $|i-j| = 1$. We call the vertices $p_1$ and $p_k$ the \emph{ends of $P$}, and say that $P$ is \emph{from $p_1$ to $p_k$}. The \emph{interior of $P$}, denoted by $P^*$, is the set $P \setminus \{p_1, p_k\}$. The \emph{length} of a path is its number of edges (so a path of length at most one has empty interior). Similarly, if $C$ is a cycle, we write $C = c_1 \dd \cdots \dd c_k\dd c_1$ to mean that $V(C) = \{c_1, \dots, c_k\}$ and $c_i$ is adjacent to $c_j$ if $|i-j|\in \{1,k-1\}$. The \textit{length} of a cycle is its number edges (or equivalently, vertices.)

A {\em theta} is a graph $\Theta$ consisting of two non-adjacent vertices $a, b$, called the \textit{ends of $\Theta$}, and three pairwise internally disjoint paths $P_1, P_2, P_3$ from $a$ to $b$ of length at least two, called the \textit{paths of $\Theta$}, such that $P_1^*, P_2^*, P_3^*$ are pairwise anticomplete to each other. For a graph $G$, by a \textit{theta in $G$} we mean an induced subgraph of $G$ which is a theta.

A {\em prism} is a graph $\Pi$ consisting of two disjoint triangles $\{a_1,a_2,a_3\}, \{b_1,b_2,b_3\}$ called the \textit{triangles of $\Pi$}, and three pairwise disjoint paths $P_1,P_2,P_3$ called the \textit{paths of $\Pi$}, where $P_i$ has ends $a_i,b_i$
for each $i\in \{1,2,3\}$, and for distinct $i,j\in \{1,2,3\}$, $a_ia_j$ and $b_ib_j$ are the only edges between $P_i$ and $P_j$. For a graph $G$, by a \textit{prism in $G$} we mean an induced subgraph of $G$ which is a prism.

A {\em pyramid} is a graph $\Sigma$ consisting of a vertex $a$, a triangle $\{b_1, b_2, b_3\}$ and three paths $P_1,P_2,P_3$ of length at least one with $P_i$ from $a$ to $b_i$ for each $i\in \{1,2,3\}$ and otherwise pairwise disjoint, such that for distinct $i,j\in \{1,2,3\}$, $b_ib_j$ is the only edge between $P_i \setminus \{a\}$ 
and $P_j \setminus \{a\}$, and at most one of $P_1, P_2, P_3$ has
length exactly one. We say that $a$ is the {\em apex} of $\Sigma$, $b_1b_2b_3$ is the {\em base} of $\Sigma$, and $P_1,P_2,P_3$ are the \textit{paths} of $\Sigma$. The pyramid $\Sigma$ is said to be \textit{long} if all its paths have lengths more than one. For a graph $G$, by a \textit{pyramid in $G$} we mean an induced subgraph of $G$ which is a pyramid.

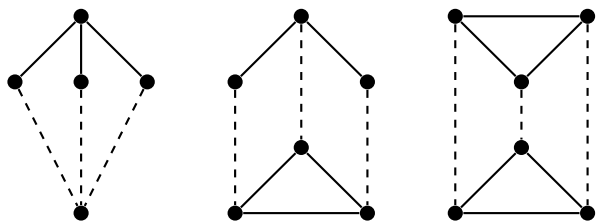
\begin{figure}[ht]
\label{fig:3PCs}
\begin{center}
\begin{tikzpicture}[scale=0.29]

\node[inner sep=2pt, fill=black, circle] at (0, 3)(v1){}; 
\node[inner sep=2pt, fill=black, circle] at (-3, 0)(v2){}; 
\node[inner sep=2pt, fill=black, circle] at (3, 0)(v3){}; 
\node[inner sep=2pt, fill=black, circle] at (0, 0)(v4){}; 
\node[inner sep=2pt, fill=black, circle] at (0, -6)(v5){};

\node[inner sep=2pt, fill=white, circle] at (0, -4.8)(v21){};

\draw[black, thick] (v1) -- (v2);
\draw[black, thick] (v1) -- (v3);
\draw[black, thick] (v1) -- (v4);
\draw[black, dashed, thick] (v2) -- (v5);
\draw[black, dashed, thick] (v3) -- (v5);
\draw[black, dashed, thick] (v4) -- (v5);

\end{tikzpicture}
\hspace{0.7cm}
\begin{tikzpicture}[scale=0.29]

\node[inner sep=2pt, fill=black, circle] at (0, 3)(v1){}; 
\node[inner sep=2pt, fill=black, circle] at (-3, 0)(v2){}; 
\node[inner sep=2pt, fill=black, circle] at (3, 0)(v3){}; 
\node[inner sep=2pt, fill=black, circle] at (-3, -6)(v4){}; 
\node[inner sep=2pt, fill=black, circle] at (0, -3)(v5){}; 
\node[inner sep=2pt, fill=black, circle] at (3, -6)(v6){};

\node[inner sep=2pt, fill=white, circle] at (0, -4.8)(v21){}; 

\draw[black, thick] (v1) -- (v2);
\draw[black, thick] (v1) -- (v3);
\draw[black, dashed, thick] (v1) -- (v5);
\draw[black, dashed, thick] (v2) -- (v4);
\draw[black, dashed, thick] (v3) -- (v6);
\draw[black, thick] (v4) -- (v5);
\draw[black, thick] (v4) -- (v6);
\draw[black, thick] (v5) -- (v6);

\end{tikzpicture}
\hspace{0.7cm}
\begin{tikzpicture}[scale=0.29]

\node[inner sep=2pt, fill=black, circle] at (-3, 3)(v1){}; 
\node[inner sep=2pt, fill=black, circle] at (0, 0)(v2){}; 
\node[inner sep=2pt, fill=black, circle] at (3, 3)(v3){}; 
\node[inner sep=2pt, fill=black, circle] at (-3, -6)(v4){}; 
\node[inner sep=2pt, fill=black, circle] at (0, -3)(v5){}; 
\node[inner sep=2pt, fill=black, circle] at (3, -6)(v6){};

\node[inner sep=2pt, fill=white, circle] at (0, -4.8)(v21){}; 

\draw[black, thick] (v1) -- (v2);
\draw[black, thick] (v1) -- (v3);
\draw[black, thick] (v2) -- (v3);
\draw[black, dashed, thick] (v1) -- (v4);
\draw[black, dashed, thick] (v2) -- (v5);
\draw[black, dashed, thick] (v3) -- (v6);
\draw[black, thick] (v4) -- (v5);
\draw[black, thick] (v4) -- (v6);
\draw[black, thick] (v5) -- (v6);

\end{tikzpicture}
\end{center}
\caption{Theta, pyramid and prism. The dashed lines represent paths of length at least one.}
\label{fig:forbidden_isgs}
\end{figure}

Let us now mention a few results from the literature which we will use in this paper.
Let $G$ be a graph. By a \textit{separation} in $G$ we mean a triple $(L,M,R)$ of pairwise disjoint subsets of vertices in $G$ with $L\cup M\cup R=G$, such that neither $L$ nor $R$ is empty and $L$ is anticomplete to  $R$ in $G$. Let $x,y\in G$ be distinct. We say a set $M\subseteq G\setminus \{x,y\}$ \textit{separates $x$ and $y$} if there exists a separation $(L,M,R)$ in $G$ with $x\in L$ and $y\in R$. Also, for disjoint sets $X,Y\subseteq G$, we say a set $M\subseteq G\setminus (X\cup Y)$ \textit{separates $X$ and $Y$} if there exists a separation $(L,M,R)$ in $G$ with $X\subseteq L$ and $Y\subseteq R$. If $X=\{x\}$, we say that \textit{$M$ separates $x$ and $Y$} to mean $M$ separates $X$ and $Y$. Recall the following well-known theorem of Menger \cite{Menger}:
\begin{theorem}[Menger \cite{Menger}]\label{Menger}
   Let $k\geq 1$ be an integer, let $G$ be a graph and let $x,y\in G$ be distinct and non-adjacent. Then either there exists a set $M\subseteq G\setminus \{x,y\}$ with $|M|<k$ such that $M$ separates $x$ and $y$, or there are $k$ pairwise internally disjoint paths in $G$ from $x$ to $y$.
\end{theorem}
 Let $k$ be a positive integer and let $G$ be a graph. A \textit{strong $k$-block} in $G$ is a set $B$ of at least $k$ vertices in $G$ such that for every $2$-subset $\{x,y\}$ of $B$, there exists a collection $\mathcal{P}_{\{x,y\}}$ of at least $k$ distinct and pairwise internally disjoint paths in $G$ from $x$ to $y$, where for every two distinct $2$-subsets $\{x,y\}, \{x',y'\}\subseteq B$ of $G$, and every choice of paths $P\in \mathcal{P}_{\{x,y\}}$ and $P'\in \mathcal{P}_{\{x',y'\}}$, we have $P\cap P'=\{x,y\}\cap \{x',y'\}$.

For a tree $T$ and $xy\in E(T)$, we denote by $T_{x,y}$ the component of $T-xy$ containing $x$. Let $G$ be a graph and $(T,\chi)$ be a tree decomposition for $G$. For every $S\subseteq T$, let $\chi(S)=\bigcup_{x\in S}\chi(x)$. By an \textit{adhesion} of $(T,\chi)$ we mean the set $\chi(x)\cap \chi(y)=\chi(T_{x,y})\cap \chi(T_{y,x})$ for some $xy\in E(T)$. For every $x\in V(T)$, by the \textit{torso at $x$}, denoted by $\hat{\chi}(x)$, we mean the graph obtained from the bag $\chi(x)$ by, for each $y\in N_T(x)$, adding an edge between every two non-adjacent vertices $u,v\in \chi(x,y)$. In \cite{twvii}, we used Theorem~\ref{boundeddegree} and the following result from \cite{tighttw}:

\begin{theorem}[Erde and Wei\ss auer \cite{tighttw}, see also \cite{Grohe}]\label{tightdegorg}
Let $r$ be a positive integer, and let $G$ be a graph containing no subdivision of $K_r$ as a subgraph. Then $G$ admits a tree decomposition $(T,\chi)$ for which the following hold.
\begin{itemize}
\item Every adhesion of $(T,\chi)$ has cardinality less than $r^2$.
\item For every $x\in V(T)$, either $\hat{\chi}(x)$ has fewer than $r^2$ vertices of degree at least $2r^4$, or $\hat{\chi}(x)$ has no minor isomorphic to  $K_{2r^2}$.
\end{itemize}
\end{theorem}

to prove the following.

\begin{theorem}[Abrishami, Alecu, Chudnovsky, Hajebi and Spirkl \cite{twvii}]\label{noblocksmalltw_wall}
  Let $k,t\geq 1$ be integers. Then there exists an integer $w=w(k,t)\geq 1$ such that every $t$-clean graph with no strong $k$-block has treewidth at most $w$.
\end{theorem}
Note that for every $t \geq 3$, every subdivision of $W_{t \times t}$ contains a theta and 
the line graph of every subdivision of $W_{t \times t}$ contains a prism. It follows that for every $t\geq 1$, every graph in $\mathcal{C}_t$ is $t$-clean, and so the following is
immediate from Theorem~\ref{noblocksmalltw_wall}:

\begin{corollary}\label{noblocksmalltw_Ct}
For all integers $k,t\geq 1$, there exists an integer $\beta=\beta(k,t)$ such that every graph in $\mathcal{C}_t$ with no strong $k$-block has treewidth at most
  $\beta(k,t)$.
  \end{corollary}

A vertex $v$ in a graph $G$ is said to be a \textit{branch vertex} if $v$ has degree more than two. By a {\em caterpillar} we mean a tree $C$ with maximum degree three such that there is a path $P$ in $C$ containing all branch vertices of $C$ (our definition of a caterpillar is non-standard for two reasons: a caterpillar is often allowed to be of arbitrary maximum degree, and the path $P$ from the definition often contains all vertices of degree more than one). By a \textit{subdivided star} we mean a graph isomorphic to a subdivision of the complete bipartite graph $K_{1,\delta}$ for some $\delta\geq 3$. In other words, a subdivided star is a tree with exactly one branch vertex, which we call its \textit{root}. For every graph $H$, a vertex $v$ of $H$ is said to be \textit{simplicial} if $N_H(v)$ is a clique. We denote by $\mathcal{Z}(H)$ the set of all simplicial vertices of $H$. Note that for every tree $T$, $\mathcal{Z}(T)$ is the set of all leaves of $T$. An edge $e$ of a tree $T$ is said to be a \textit{leaf-edge} of $T$ if $e$ is incident with a leaf of $T$. It follows that if $H$ is the line graph of a tree $T$, then $\mathcal{Z}(H)$ is the set of all vertices in $H$ corresponding to the leaf-edges of $T$. The following is proved in \cite{twvii} based on (and refining) a result from \cite{Davies}. 
\begin{theorem}[Abrishami, Alecu, Chudnovsky, Hajebi and Spirkl \cite{twvii}]\label{connectifier}
  For every integer $h\geq 1$, there exists an integer $\mu=\mu(h)\geq 1$ with the following property. Let $G$ be a connected graph with no clique of cardinality $h$ and let
  $S \subseteq G$ such that $|S|\geq \mu$. Then either some path in $G$ contains $h$ vertices from $S$, or 
  there is an induced subgraph $H$ of $G$ with $|H\cap S|=h$ for which one of the following holds.
  \begin{itemize}
  \item $H$ is either a caterpillar or the line graph of a caterpillar with $H\cap S=\mathcal{Z}(H)$.
  \item $H$ is a subdivided star with root $r$ such that $\mathcal{Z}(H)\subseteq H\cap S\subseteq \mathcal{Z}(H)\cup \{r\}$.
\end{itemize}
  \end{theorem}
\section{Jumps and jewels on pyramids with trapped apices}
\label{sec:pyramids}
For a graph $G$, an induced subgraph $H$ of $G$ and a vertex $a\in H$, we say $a$ is \textit{trapped} in $H$ if
\begin{itemize}
    \item we have $N^2_G[a]\subseteq H$; and
    \item every vertex in $N_H(a)=N_G(a)$ has degree two in $H$ (and so in $G$).
\end{itemize}
The goal of this section is, for a graph $G\in \mathcal{C}$, $H\subseteq G$ and a pyramid $\Sigma$ in $H$, to investigate the adjacency between $\Sigma$ and a path in $G \setminus H$, assuming that the apex of $\Sigma$ is trapped in $H$. This will be of essential use in the next section. 

We begin with a few definitions. Let $G$ be a graph and let $\Sigma$ be a pyramid in $G$  with apex $a$, base $b_1b_2b_3$ and
paths $P_1,P_2,P_3$. A set $X \subseteq \Sigma$ is said to be {\em local (in $\Sigma$)} if either $X \subseteq P_i$ for some $i \in \{1,2,3\}$ or $X \subseteq \{b_1,b_2,b_3\}$.   Let $P$ be a path in $G\setminus \Sigma$ with (not necessarily distinct) ends $p_1,p_2$. For $i\in \{1,2,3\}$, we say $P$ is a {\em corner path for $\Sigma$ at $b_i$} if
\begin{itemize}
    \item $p_1$ has at least one neighbor in $P_i\setminus \{b_i\}$;
    \item $p_2$ is complete to $\{b_1,b_2,b_3\}\setminus \{b_i\}$; and 
    \item except for the edges between $\{p_1,p_2\}$ and $\Sigma$ described in the above two bullets, there is no edge with an end in $P$ and an end in $\Sigma\setminus \{b_i\}$.
\end{itemize}
By a {\em corner path for $\Sigma$} we mean a corner path for $\Sigma$ at one of $b_1$, $b_2$ or $b_3$.

Let $p\in G\setminus \Sigma$. Then $p$ is said to be \textit{narrow for $\Sigma$} if $N_{\Sigma}(p)$ is local in $\Sigma$. Otherwise, we say $p$ is \textit{wide for $\Sigma$}. For $i\in \{1,2,3\}$, we say $p$ is a \textit{jewel for $\Sigma$ at $b_i$} if $p$ is anticomplete to $P_i$ (in particular, $p$ is anticomplete to $a$), and for every $j\in \{1,2,3\}\setminus \{i\}$, we have $N_{P_j}(p)=N_{P_j}[b_j]$. By a {\em jewel for $\Sigma$} we mean a jewel for $\Sigma$ at one of $b_1$, $b_2$ or $b_3$.
Note that if $p$ is either a corner path or a jewel for $\Sigma$, then $p$ is wide for $\Sigma$. The following lemma establishes a converse to this fact for graphs in $\mathcal{C}$ and pyramids with a trapped apex.
\begin{lemma}
  \label{major}
  Let $G\in \mathcal{C}$ be a graph, let $H\subseteq G$, and let $a\in V(H)$ be trapped in $H$. Let $\Sigma$ be a pyramid in $H$ with apex $a$, base $b_1b_2b_3$, and paths $P_1,P_2,P_3$. Let $p\in G\setminus H$. Then $p$ is wide for $\Sigma$ if and only if $p$ is either a corner path for $\Sigma$ or a jewel for $\Sigma$.
  \end{lemma}

\begin{proof}
We only need to prove the ``only if'' direction. Assume that $p\in G\setminus H$ is wide for $\Sigma$ and $p$ is not a corner path for $\Sigma$. Since $a$ is trapped in $H$ and $p\in G\setminus H$, it follows that $\Sigma$ is long and $p$ is anticomplete to $N_{\Sigma}[a]$. First, we show that:

\sta{\label{atleastoneanti} There exists $i\in \{1,2,3\}$ for which $p$ is anticomplete to $P_i$.}

Suppose for a contradiction that $p$ has a neighbor in each of $P_1,P_2,P_3$. Since $p$ is wide for $\Sigma$ and $p$ is not a corner path for $\Sigma$, we may assume without loss of generality that $p$ has a neighbor in $P_1^*$ and a neighbor in $P_2^*$. For each $i\in \{1,2,3\}$, traversing $P_i$ from $a$ to $b_i$, let $x_i$ be the first neighbor of $p$ in $P_i$. Since $a$ is trapped, it follows that $x_1\in P_1^*$, $x_2\in P_2^*$ and $x_3\in P_3\setminus N_{\Sigma}[a]$. But then $G$ contains a theta with ends $a, p$ and paths $a\dd P_i\dd x_i\dd p$ for $i\in \{1,2,3\}$, a contradiction. This proves \eqref{atleastoneanti}.\medskip
  
By \eqref{atleastoneanti} and without loss of generality, we may assume that $p$ is anticomplete to $P_3$. Note that since $p$ is wide for $\Sigma$, it follows that for every $j\in \{1,2\}$, $p$ has a neighbor in $P_j$, and there exists $j\in \{1,2\}$ for which $p$ has a neighbor in $P_j^*$. For each $j\in \{1,2\}$, traversing $P_j$ from $a$ to $b_j$, let $x_j$ and $y_j$ be the first and the last neighbor of $p$ in $P_j$, respectively. Then we have $x_j\in P_j^*\setminus N_{P_j}(a)$ for some $j\in \{1,2\}$.  In fact, the following holds.

\sta{\label{notbi} For every $j\in \{1,2\}$, we have $x_j\in P_j^*\setminus N_{P_j}(a)$.}
 
Suppose not. Since $p$ is wide for $\Sigma$ and $p$ is anticomplete to $N_{\Sigma}(a)$, we may assume without loss of generality that $p$ has a neighbor in $P_1^*$ and $x_2=y_2=b_2$. But now $G$ contains a theta with ends $a,b_2$ and paths  $a \dd P_1\dd x_1\dd p\dd b_2$, $a \dd P_2 \dd b_2$ and $a\dd P_3\dd b_3\dd b_2$, a contradiction. This proves \eqref{notbi}.

\sta{\label{clique2} For every $j\in \{1,2\}$, $N_{P_j}(p)$ is a clique of cardinal ity two.}

Suppose not. Then we may assume without loss of generality that either $x_1=y_1$ or $x_1$ and $y_1$ are distinct and non-adjacent. By \eqref{notbi}, for every $j\in \{1,2\}$, we have $x_j\in P_j^*\setminus N_{P_j}(a)$. Therefore, if $x_1=y_1$, then $G$ contains a theta with ends $a,x_1$ and paths $a \dd P_1\dd x_1$, $a \dd P_2\dd x_2\dd p\dd x_1$ and $a\dd P_3\dd b_3\dd b_1\dd P_1\dd x_1$, which is impossible. Thus, $x_1$ and $y_1$ are distinct and non-adjacent. But now $G$ contains a theta with ends $a,p$ and paths $a \dd P_1\dd x_1\dd p$, $a \dd P_2\dd x_2\dd p$ and $a\dd P_3\dd b_3\dd b_1\dd P_1\dd y_1\dd p$, a contradiction. This proves \eqref{clique2}.\medskip

The proof is almost concluded. By \eqref{clique2}, for every $j\in \{1,2\}$, we have $N_{P_j}(p)=\{x_j,y_j\}$ and $x_j$ is adjacent to $y_j$. If $y_j\in P_j^*$ for some $j\in \{1,2\}$, then $G$ contains a prism with triangles $x_jy_jp$ and $b_1b_2b_3$ and paths $x_j\dd P_j\dd a\dd P_3\dd b_3$, $y_j\dd P_j\dd b_j$ and $p\dd y_{3-j}\dd P_{3-j}\dd b_{3-j}$, a contradiction. Hence, we have $y_j=b_j$ for every $j\in \{1,2\}$, and so $p$ is a jewel corner for $\Sigma$ at $b_i$. This completes the proof of Lemma~\ref{major}.
 \end{proof}

We can now prove the main result of this section.
\begin{theorem}\label{pyramid2}
  Let $G\in \mathcal{C}$ be a graph, let $H\subseteq G$, and let $a\in V(H)$ be trapped in $H$. Let $\Sigma$ be a pyramid in $H$ with apex $a$, base $b_1b_2b_3$, and paths $P_1,P_2,P_3$. Let $P$ be a path in $G \setminus H$. Then one of the following holds. 
\begin{itemize}
\item $N_{\Sigma}(P)$ is local in $\Sigma$.
\item $P$ contains a corner path for $\Sigma$.
\item $P$ contains a jewel for $\Sigma$.
\end{itemize}
\end{theorem}

\begin{proof}
 Suppose for a contradiction that there exists a path $P$ in $G\setminus H$ for which none of the outcomes of Theorem~\ref{pyramid2} hold. We choose such a path $P$ with $|P|$ minimum. It follows that $N_{\Sigma}(P)$ is not local in $\Sigma$, $N_{\Sigma}(X)$ is local in $\Sigma$ for every connected set $X\subsetneq P$, $P$ contains no corner path for $\Sigma$ and $P$ contains no jewel for $\Sigma$. Therefore, by Lemma~\ref{major}, we have $|P|>1$. Since $a$ is trapped in $H$ and $P\subseteq G\setminus H$, it follows that $\Sigma$ is long and $P$ is anticomplete to $N_{\Sigma}[a]$. For every $i\in \{1,2,3\}$, let $P_i'=P_i\setminus N_{P_i}[a]$. Since $N_{\Sigma}(P)$ is not local and $P$ is minimal subject to this property, we may assume without loss of generality that
  \begin{itemize}
      \item 
$N_{\Sigma}(p_1)  \subseteq P'_1$ and $p_1$ has a neighbor in $P'_1 \setminus \{b_1\}$; and 
    \item $p_2$ has a neighbor in $P_2'$, and either $N_{\Sigma}(p_2) \subseteq P_2'$ or $N_{\Sigma}(p_2) \subseteq \{b_1,b_2,b_3\}$.
  \end{itemize}
  It follows from the choice of $P$ that $P^*$ is anticomplete to $\Sigma\setminus \{b_1\}$. For each $i\in \{1,2\}$, traversing $P_i$ from $a$ to $b_i$, let $x_i$ and $y_i$ be the first and the last neighbor of $p_i$ in $P_i$, respectively. So we have $x_1\in P_1'\setminus \{b_1\}$, $y_1\in P_1'$ and $x_2,y_2\in P_2'$. In fact, the following holds.
 
\sta{\label{onlyb2} We have $x_2\in P_2'\setminus \{b_2\}$.}
  
Suppose not. Then we have $x_2=y_2=b_2$, and so $b_2\in N_{\Sigma}(p_2)\subseteq \{b_1,b_2,b_3\}$. Since $P$ is not a corner path for $\Sigma$ at $b_1$, it follows that $p_2$ is not adjacent to $b_3$. But now $G$ contains a theta with ends $a,b_2$ and paths $a\dd P_1\dd x_1\dd p_1\dd P\dd p_2\dd b_2$, $a\dd P_2\dd b_2$ and $a\dd P_3\dd b_3\dd b_2$, a contradiction. This proves \eqref{onlyb2}.

\medskip

  In view of \eqref{onlyb2} and the choice of $P$, we conclude that $P^*$ is anticomplete to $\Sigma$, and for every $i\in \{1,2\}$, we have $N_{\Sigma}(p_i)=N_{P_i'}(p_i)$, $x_i\in P_i'\setminus \{b_i\}$ and $y_i\in P_i'$.

\sta{\label{antitoP3} For every $i\in \{1,2\}$, $x_i$ and $y_i$ are distinct and adjacent.}

Suppose not. Then we may assume without loss of generality that either $x_1=y_1$ or $x_1$ and $y_1$ are distinct and non-adjacent. In the former case, $G$ contains a theta with ends $a,x_1$ and paths $a \dd P_1\dd x_1$, $a \dd P_2\dd x_2\dd p_2\dd P\dd p_1\dd x_1$ and $a\dd P_3\dd b_3\dd b_1\dd P_1\dd x_1$, a contradiction. It follows that $x_1$ and $y_1$ are distinct and non-adjacent. But then $G$ contains a theta with ends $a,p_1$ and paths $a \dd P_1\dd x_1\dd p_1$, $a \dd P_2\dd x_2\dd p_2\dd P\dd p_1$ and $a\dd P_3\dd b_3\dd b_1\dd P_1\dd y_1\dd p_1$, again a contradiction. This proves \eqref{antitoP3}.\medskip

By \eqref{antitoP3}, for every $i\in \{1,2\}$, we have $N_{P_i}(p)=\{x_i,y_i\}$ and $x_i$ is adjacent to $y_i$. But now $G$ contains a $G$ contains a prism with triangles $p_1x_1y_1$ and $p_2x_2y_2$ and paths $P$, $x_1\dd P_1\dd a\dd P_2\dd x_2$ and $y_1\dd P_1\dd b_1\dd b_2\dd P_2\dd y_2$, a contradiction. This completes the proof of Theorem~\ref{pyramid2}.
\end{proof}

\section{Strip structures with an ornament of jewels} \label{sec:strips}

The main result of this section, Theorem~\ref{stripstructure}, provides a description of the structure of graphs in $\mathcal{C}$ which contain a pyramid with a trapped apex. 

We first set up a framework that allows us to think of a pyramid with apex $a$ as a special case of a construction similar to the line graph of a tree $T$, which we call a
``$(T,a)$-strip-structure.'' We start with an induced subgraph $W$ of $G$ that admits an ``optimal'' $(T,a)$-strip-structure in $G$ in a certain sense, and show that the rest of the graph fits into the same construction, except for vertices which are jewels for certain canonically positioned pyramids in $W$.

First, we need to properly define a strip-structure (this is similar to \cite{SPGT},
 \cite{bisimp2}, and \cite{threeinatree}).
A tree $T$ is said to be {\em smooth} if $T$ has at least three vertices and every vertex of $T$ is either a branch vertex or a leaf.
Let $G$ be a graph, let $a\in G$, let $T$ be a smooth tree, and
let $\eta: V(T) \cup E(T) \cup (E(T) \times V(T))\rightarrow 2^{G \setminus \{a\}}$ be a function. For every $S\subseteq V(T)$, we define $\eta(S)=\bigcup_{v \in S, e \in E(T[S])} (\eta(v) \cup \eta(e))$ and $\eta^+(S)=\eta(S)\cup \{a\}$.
For every vertex $v\in V(T)$, we define $B_\eta(v)$ to be the union of all sets $\eta(e,v)$ taken over all edges $e\in E(T)$ incident with $v$ (we often omit the subscript $\eta$ unless there is ambiguity).

The function $\eta$ is said to be a \textit{$(T,a)$-strip-structure in $G$} if the following conditions are satisfied.
\begin{enumerate}[(S1)]
\item \label{s1} For all distinct $o,o' \in V(T)\cup E(T)$, we have $\eta(o) \cap \eta(o')= \emptyset$.
  \item \label{s2}  If $l\in V(T)$ is a leaf of $T$, then $\eta(l)$ is empty.
  \item \label{s3}  For all $e \in E(T)$ and $v \in V(T)$, we have $\eta(e,v) \subseteq \eta(e)$, and $\eta(e,v)\neq \emptyset$ if and only if $e$ is incident with $v$.
\item \label{s4} For all distinct edges $e,f \in E(T)$ and every vertex $v \in V(T)$, $\eta(e,v)$ is complete to $\eta(f,v)$, and there are no other edges between $\eta(e)$ and $\eta(f)$. In particular, if $e$ and $f$ share no end, the $\eta(e)$ is anticomplete to $\eta(f)$.
 \item \label{s5}  For every $e\in E(T)$ with ends $u,v$, define $\eta^{\circ}(e)=\eta(e)\setminus (\eta(e,u)\cup \eta(e,v))$. Then for every vertex $x \in \eta(e)$, there is a path in $\eta(e)$ from $x$ to a vertex in $\eta(e,u)$ with interior contained in $\eta^{\circ}(e)$, and there is a path in $\eta(e)$ from $x$ to a vertex in $\eta(e,v)$ with interior contained in $\eta^{\circ}(e)$. 
\item \label{s6} For all $v \in V(T)$ and $e \in E(T)$, $\eta(v)$ is anticomplete to
  $\eta(e) \setminus \eta(e,v)$. In other words, we have $N_{\eta(T)}(\eta(v))\subseteq B_{\eta}(v)$.
    \item \label{s7} For every $v\in V(T)$ and every connected component $D$ of $\eta(v)$, we have $N_{B_{\eta}(v)}(D)\neq \emptyset$.
\item \label{s8} For every leaf $l\in V(T)$ of $T$, assuming $e\in E(T)$ to be the leaf-edge of $T$
  incident with $l$, $a$ is complete to $\eta(e,l)$. Also, $a$ has no other neighbors in $\eta(T)$.
\end{enumerate}

Let $S\subseteq \eta(T)$. We say that $S$
  is {\em local in $\eta$} if $ S\subseteq \eta(e)$ for some
  $e \in E(T)$ or $S \subseteq B_\eta(v) \cup \eta(v)$ for some $v \in V(T)$. The following lemma shows that every non-local subset contains a $2$-subset (that is, a subset of cardinality two) which is non-local.
  
\begin{lemma}\label{twoattachments}
 Let $G$ be a graph and $a \in V(G)$. Let $T$ be a smooth tree and $\eta$ be a $(T,a)$-strip-structure in $G$. Assume that $C \subseteq \eta(T)$ is not local in $\eta$. Then there is a $2$-subset of $C$ which is not local in $\eta$.
\end{lemma} 
  \begin{proof}
      
  First, suppose there exists a vertex 
  $x \in C \cap \eta^{\circ}(e)$
  for some $e\in E(T)$.
  Since $C$ is not local, there exists
  $y\in C \setminus \eta(e)$. Now $\{x,y\}$ is a $2$-subset of $C$ which is not local in $\eta$, as desired. Therefore, we may assume that $C\subseteq \bigcup_{v\in V(T)} (B(v)\cup \eta(v))$. Since the empty set is local in $\eta$, we have $C\neq \emptyset$; thus, we may pick $x\in C$, $v\in V(T)$ and $e\in E(T)$ such that $x\in \eta(e,v)\cup \eta(v)$. If there exists a vertex 
  $y\in C \setminus (\eta(e) \cup B(v)\cup \eta(v))$, then $\{x,y\}$ is a $2$-subset of $C$ which is not local in $\eta$, and so we are done. Consequently, we may assume that $C \subseteq \eta(e) \cup B(v)\cup \eta(v)$. Since $C$ is not local, there exist $x' \in \eta(e) \setminus (B(v)\cup \eta(v)))$
  and $y' \in (B(v)\cup \eta(v)) \setminus \eta(e)$ such that $\{x',y'\}\subseteq C$. Now $\{x',y'\}$ is a $2$-subset of $C$ which is not local in $\eta$, as required. This completes the proof of Lemma~\ref{twoattachments}.
   \end{proof}

In order to state and prove the main result of this section, we need to define several notions related to strip-structures. From here until the statement of Theorem~\ref{stripstructure}, let us fix a graph $G$, a vertex $a\in G$, a smooth tree $T$, and a $(T,a)$-strip-structure $\eta$ in $G$.

For every edge $e\in E(T)$ with ends $u,v$, by an {\em $\eta(e)$-rung}, we mean a path $P$ in $\eta(e)\subseteq \eta(T)$ for which either $|P|=1$ and $P\subseteq \eta(e,u) \cap \eta(e,v)$, or $P$ has an end in $\eta(e,u) \setminus \eta(e,v)$, an end in $\eta(e,v) \setminus \eta(e,u)$, and $P^*\subseteq \eta^{\circ}(e)$. Equivalently, a path $P$ in  $\eta(e)$ is an $\eta(e)$-rung if $P$ has an end in $\eta(e,u)$, an end in $\eta(e,v)$, and $|P \cap \eta(e, u)| = |P \cap \eta(e, v)| = 1$. 
It follows from (S\ref{s5}) that every vertex in $\eta(e)\setminus \eta^{\circ}(e)$ is contained in an $\eta(e)$-rung. In particular, if either $\eta(e,u)\subseteq  \eta(e,v)$ or $\eta(e,v)\subseteq \eta(e,u)$, then $\eta(e,u)=\eta(e,v)$. An $\eta(e)$-rung is said to be \textit{long} if it is of non-zero length.

For every edge $e \in E(T)$, let $\tilde{\eta}(e)$ be the set of vertices in
$\eta(e)$ that are not in any $\eta(e)$-rung (so $\tilde{\eta}(e)\subseteq \eta^{\circ}(e)$.) We say that $\eta$ is {\em tame}
if
\begin{itemize}
\item $\eta(v) = \emptyset$ for every $v \in V(T)$; and
\item $\tilde{\eta}(e)=\emptyset$ for every $e \in E(T)$.
\end {itemize}
In other words, $\eta$ is tame if and only if every vertex in
$\eta(T)$ is in an $\eta(e)$-rung for some $e\in E(T)$.

For a $(T,a)$-strip-structure $\eta'$ in $G$, we write $\eta\leq \eta'$ to mean that for every $o\in V(T)\cup E(T)\cup (E(T)\times V(T))$, we have $\eta(o)\subseteq \eta'(o)$.  We say that a $(T,a)$-strip-structure $\eta$ is \textit{substantial} if for every $e\in E(T)$, there exists a long $\eta(e)$-rung in $G$. Equivalently, $\eta$ is substantial if for every edge $e\in E(T)$ with ends $u,v$, we have $\eta(e,u)\neq \eta(e,v)$, and so $\eta(e,u)\setminus \eta(e,v),\eta(e,v)\setminus \eta(e,u)\neq \emptyset$. One may observe that since $T$ has at least three vertices, if $\eta$ is substantial and $\eta\leq \eta'$, then $\eta'$ is substantial too.

We say $\eta$ is \textit{rich} if
\begin{itemize}
    \item $a$ is trapped in $\eta^+(T)$; and
    \item for every leaf $l\in V(T)$ of $T$, assuming $e\in E(T)$ to be the leaf-edge of $T$ incident with $l$, we have $|\eta(e,l)|=1$.
\end{itemize}
It follows that if there exists a rich $(T,a)$-strip-structure $\eta$ in $G$, then $T$ has exactly $|N_G(a)|$ leaves, and for every leaf $l\in V(T)$ of $T$, assuming $e\in E(T)$ to be the leaf-edge of $T$ incident with $l$ and $v\in V(T)$ to be the unique neighbor of $l$ in $T$, we have $\eta(e,v)\cap \eta(e,l)=\emptyset$.

By a \textit{seagull in $T$} we mean a triple $(v,e_1,e_2)$ where $v\in V(T)$ and $e_1,e_2$ are two distinct edges of $T$ incident with $v$. By a \textit{claw in $T$} we mean a $4$-tuple $(v,e_1,e_2,e_3)$ where $v\in V(T)$ and $e_1,e_2,e_3$ are three distinct edges of $T$ incident with $v$.
        
        Let $(v,e_1,e_2,e_3)$ be a claw in $T$. By an \textit{$\eta$-pyramid at $(v,e_1,e_2,e_3)$}, we mean a pyramid $\Sigma$ with apex $a$, base $b_1b_2b_3$ and paths $P_1,P_2,P_3$, satisfying the following for each $i\in \{1,2,3\}$.
        \begin{itemize}
            \item  $b_i\in \eta(e_i,v)$.
            \item There exists a leaf $l_i$ of $T$ with the following properties:
              \begin{enumerate}
              \item $l_i$  belongs to the component of $T\setminus \{e_i\}$ not containing $v$.
              \item Let $\Lambda_i$ be the unique path in $T$ from $v$ to $l_i$ (so $e_i\in E(\Lambda_i)$). Then  $P_i=\Gamma_i\cup \{a\}$, where $\Gamma_i$ is a path in $\bigcup_{e\in E(\Lambda_i)}\eta(e)$ such that $R_i=\Gamma_i\cap \eta(e_i)$ is a long $\eta(e_i)$-rung and $\Gamma_i\cap \eta(e)$ is a $\eta(e)$-rung for each $e\in E(\Lambda_i)\setminus \{e_i\}$.
                \end{enumerate}
        \end{itemize}
        In particular, assuming $u_i$ to be the ends of $e_i$ distinct from $v$ and $c_i$ to be the unique vertex in $N_{R_i}(b_i)=N_{P_i}(b_i)$ for each $i\in \{1,2,3\}$, we have $b_i\in \eta(e_i,v)\setminus \eta(e_i,u_i)$ and $c_i\in \eta(e_i)\setminus \eta(e_i,v)$.
        
For a branch vertex $v\in V(T)$, by an \textit{$\eta$-pyramid at $v$} we mean an $\eta$-pyramid at $(v,e_1,e_2,e_3)$ for some claw $(v,e_1,e_2,e_3)$ in $T$. Also, by an \textit{$\eta$-pyramid} we mean an $\eta$-pyramid at $v$ for some branch vertex $v\in V(T)$. It follows that every $\eta$-pyramid is a long pyramid (recall that a pyramid is long if all its paths have lengths more than one). Also, if $\eta$ is substantial, then for every claw $(v,e_1,e_2,e_3)$ in $T$ there is a $\eta$-pyramid at $(v,e_1,e_2,e_3)$. 

Let $(v,e_1,e_2)$ be a seagull in $T$. A vertex $p\in G\setminus \eta^+(T)$ is said to be a \textit{jewel for $\eta$ at $(v,e_1,e_2)$} if for some edge $e_3\in E(T)\setminus \{e_1,e_2\}$ incident with $v$, there exists an $\eta$-pyramid $\Sigma$ at $(v,e_1,e_2,e_3)$ with base $b_1b_2b_3$ where $b_i\in \eta(e_i,v)$ for each $i\in \{1,2,3\}$, such that $p$ is a jewel for $\Sigma$ at $b_3$.  In particular, for each $i\in \{1,2\}$, $p$ is adjacent to $b_i$ and the unique vertex $c_i$ in $N_{P_i}(b_i)$. Therefore, since $\Sigma$ is an $\eta$-pyramid at $(v,e_1,e_2,e_3)$, assuming $u_i$ to be the end of $e_i$ distinct from $v$, it follows that $p$ has a neighbor $b_i\in \eta(e_i,v)\setminus \eta(e_i,u_i)$ and a neighbor $c_i\in \eta(e_i)\setminus \eta(e_i,v)$.

For a vertex $v\in V(T)$, by a \textit{jewel for $\eta$ at $v$} we mean a jewel for $\eta$ at $(v,e_1,e_2)$ for some seagull $(v,e_1,e_2)$ in $T$. Also, by a  \textit{jewel for $\eta$} we mean a jewel for $\eta$ at $v$ for some branch vertex $v\in V(T)$. We denote by $\mathcal{J}_{\eta}$ the set of all jewels for $\eta$. It follows that $\mathcal{J}_{\eta}\subseteq G\setminus \eta^+(T)$.

We are now in a position to prove the main result of this section:

\begin{theorem}\label{stripstructure}
Let $G \in \mathcal{C}$, let $a \in V(G)$ and let $T$ be a smooth tree. Suppose that there exists a tame, substantial, and rich $(T, a)$-strip-structure in $G$. Then there is a substantial and rich $(T,a)$-strip-structure $\zeta$ in $G$ for which $G\setminus (\zeta^+(T)\cup \mathcal{J}_{\zeta})$ is anticomplete to $\zeta^+(T)$.
  \end{theorem}  
\begin{proof}
  Let $\eta$ be a tame, substantial, and rich $(T,a)$-strip-structure in $G$ such that $\eta(T)$ is maximal with respect to inclusion. Let $M=G\setminus (\eta^+(T)\cup \mathcal{J}_{\eta})$.
  
\sta{ \label{wideclaimpath}
 Let $P$ be a path in $M$ with ends $p_1$ and $p_2$ such that there exist $x_1\in N_{\eta(T)}(p_1)$ and $x_2\in N_{\eta(T)}(p_2)$ for which $\{x_1,x_2\}$ is not local in $\eta$, and such that $|P|\geq 1$ is minimum subject to this property. Then there exists $\{j_1,j_2\}=\{1,2\}$ and $f=v_1v_2\in E(T)$ such that $x_{j_1}\in B(v_{j_1})\setminus \eta(f)$ and $x_{j_2}\in (B(v_{j_2})\cup \eta(f))\setminus B(v_{j_1})$.}

Suppose not. For each $i\in \{1,2\}$, let $e_i \in E(T)$ be such that $x_i \in \eta(e_i)$ (hence $e_1 \neq e_2$) and let $s_i$ be an end of $e_i$ such that there exists a path $\Lambda_0$ (possibly of length zero) from $s_1$ to $s_2$ in $T \setminus \{e_1,e_2\}$.  We claim that there is a vertex $v\in \Lambda_0$ such that $B(v) \cap \{x_1,x_2\} = \emptyset$. Suppose first that $s_1\neq s_2$. Let $v_1$ be the unique neighbor of $s_1$ in $\Lambda_0$. Then we have $x_1\notin B(v_1)$ and $x_2\notin B(s_1)$. Also, since $f=s_1v_1$ does not satisfy \eqref{wideclaimpath}, we have either $x_1\notin B(s_1)$ or $x_2\notin B(v_1)$. But then either $v=s_1$ or $v=v_1$ satisfies the claim. Thus, we may assume that $v=s_1=s_2$. Note that since neither $e_1$ nor $e_2$ satisfies \eqref{wideclaimpath}, we have $x_1\notin B(s_1)$ and $x_2\notin B(s_2)$. In other words, we have  $B(v) \cap \{x_1,x_2\} = \emptyset$, and the claim follows. Henceforth, let $v$ be as promised by the above claim. For each $i\in \{1,2\}$, let $u_i$ be the end of $e_i$ distinct from $s_i$ (hence $u_1\neq u_2$). Let $\Lambda=u_1\dd s_1\dd \Lambda_0\dd s_2\dd u_2$ and let $u_1',u_2'$ be the neighbors of $v$ in $\Lambda$ such that $\Lambda$ traverses $u_1,u_1',v,u_2',u_2$ in this order (so possibly $u_1=u_1'$ or $u_2=u_2'$). Let $e_i'=u_i'v$ for each $i\in \{1,2\}$. Since $T$ is smooth, there exists a vertex $u_3'\in N_{T}(v)\setminus \Lambda$. Let $e'_3=u_3'v$. For each $i\in \{1,2,3\}$, let $T_i$ be the component of $T\setminus (N_T(v)\setminus \{u_i'\})$ containing $v$ (so $e'_i\in E(T_i)$). Then since $B(v)\cap \{x_1,x_2\}=\emptyset$ and since $\eta$ is tame and substantial, there exists an $\eta$-pyramid $\Sigma$ at $(v,e'_1,e'_2,e'_3)$ with apex $a$, base $b_1b_2b_3$ and paths $P_1,P_2,P_3$ such that we have
\begin{itemize}
    \item $b_i\in \eta(e_i',v)$ and $P_i\setminus \{a,b_i\}\subseteq \eta(T_i)\setminus B(v)$ for each $i\in \{1,2,3\}$; and
    \item $x_i\in P_i^*$ for each $i\in \{1,2\}$.
\end{itemize}
In particular, the second bullet above implies that $N_{\Sigma}(P)$ is not local in $\Sigma$ and $P$ is not a corner path for $\Sigma$. Since $P\subseteq M$, we have $P\cap \mathcal{J}_{\eta}=\emptyset$. Thus, since $\Sigma$ is an $\eta$-pyramid, it follows that $P$ contains no jewel for $\Sigma$. Also, since $\eta$ is rich, $a$ is trapped in $\eta^+(T)$. Therefore, applying Theorem~\ref{pyramid2} to $G$, $H=\eta^+(T)$, $a$, $\Sigma$ and $P$, we deduce that $P$ contains a corner path for $\Sigma$. On the other hand, note that by the second bullet above, for every vertex $x\in \Sigma\setminus \{a\}$, either $\{x,x_1\}$ or $\{x,x_2\}$ is not local in $\eta$. From this, the minimality of $|P|$ and the fact that $\eta$ is rich, it follows that $P^*$ is anticomplete to $\Sigma$. But then $P$ is a corner path for $\Sigma$, a contradiction. This proves \eqref{wideclaimpath}.

\sta{\label{edgeexists} Let $P$ be a path in $M$ with ends $p_1$ and $p_2$ such that there exist $x_1\in N_{\eta(T)}(p_1)$ and $x_2\in N_{\eta(T)}(p_2)$ for which $\{x_1,x_2\}$ is not local in $\eta$, and such that $|P|\geq 1$ is minimum subject to this property. Let $f=v_1v_2\in E(T)$ and $\{j_1,j_2\}=\{1,2\}$ be as guaranteed by \eqref{wideclaimpath} applied to $P, x_1$ and $x_2$. Then we have $N_{\eta(T)}(P^*)\subseteq \eta(f,v_{j_1})$ and $N_{\eta(T)}(\{p_1,p_2\})\subseteq \eta(f)\cup B(v_1)\cup B(v_2)$.}

Suppose not. Without loss of generality, we may assume that $j_1=1$ and $j_2=2$. Note that by the minimality of $|P|$, we have $N_{\eta(T)}(P^*)\subseteq \eta(f,v_1)$. Therefore, one of $p_1$ and $p_2$ has a neighbor in $\eta(T)\setminus (\eta(f)\cup B(v_1)\cup B(v_2))$; say $p_1$ is adjacent to $x_1'\in \eta(T)\setminus (\eta(f)\cup B(v_1)\cup B(v_2))$. For each $i\in \{1,2\}$, let $T_i$ be the component of $T\setminus \{f\}$ containing $v_i$. It follows that there exists $j\in \{1,2\}$ such that $x_1'\in \eta(T_j)\setminus B(v_j)$. Assume that $|P|>1$. By the minimality of $|P|$, we have $j=1$. But then $P, x_1'$ and $x_{2}$ violate \eqref{wideclaimpath}. We deduce that $|P|=1$. But now $P, x_1'$ and $x_{3-j}$ violate \eqref{wideclaimpath}. This proves \eqref{edgeexists}.

\sta{\label{almost_one_end_complete} Let $P$ be a path in $M$ with ends $p_1$ and $p_2$ such that there exist $x_1\in N_{\eta(T)}(p_1)$ and $x_2\in N_{\eta(T)}(p_2)$ for which $\{x_1,x_2\}$ is not local in $\eta$, and such that $|P|\geq 1$ is minimum subject to this property. Suppose that there exist $\{k_1,k_2\}=\{1,2\}$,  $f=v_1v_2\in E(T)$ and $e_1\in E(T)\setminus \{f\}$ incident with $v_{k_1}$ such that $p_{k_1}$ has a neighbor in $\eta(e_1,v_{k_1})$ and $p_{k_2}$ has a neighbor in $(B(v_{k_2})\cup \eta(f))\setminus B(v_{k_1})$. Then $p_{k_1}$ is complete to $B(v_{k_1})\setminus (\eta(e_1,v_{k_1})\cup \eta(f))$.}

Due to symmetry, we may assume that $k_1=1$ and $k_2=2$. Let $e_3\in E(T)\setminus \{e_1,f\}$ be incident with $v_1$ and let $b_3\in \eta(e_3,v_1)$ be arbitrary. We need to show that $p_1$ is adjacent to $b_3$. Suppose for a contradiction that $p_1$ and $b_3$ are non-adjacent. Let $b_1\in \eta(e_1,v_1)$ be adjacent to $p_1$ and let $x\in (B(v_{2})\cup \eta(f))\setminus B(v_{1})$ be adjacent to $p_2$.  Let $T_2$ be the component of $T\setminus (N_T(v_1)\setminus \{v_2\})$ containing $v_1$ (so $f\in E(T_2)$). Also, for each $i\in \{1,3\}$, let $u_i$ be the end of $e_i$ distinct from $v_1$ and let $T_i$ be the component of $T\setminus (N_T(v_1)\setminus \{u_i\})$ containing $v_1$ (so $e_i\in E(T_i)$). By \eqref{wideclaimpath} and \eqref{edgeexists}, there exists an edge $f'=v'_1v'_2\in E(T)$ such that $N_{\eta(T)}(\{p_1,p_2\})\subseteq \eta(f')\cup B(v'_1)\cup B(v'_2)$. This, along with the minimality of $|P|$, implies that $p_1$ is anticomplete to $(\eta(T_1)\cup \eta(T_3))\setminus B(v_1)$, $P\setminus \{p_1\}$ is anticomplete to $\eta(T_1)\cup \eta(T_3)$, and $P\setminus \{p_2\}$ is anticomplete to $\eta(T_2)\setminus B(v_1)$. Since $p_2$ has a neighbor $x\in (B(v_{2})\cup \eta(f))\setminus B(v_{1})$ and since $\eta$ is tame, there exists a path $P_2$ in $G$ from $a$ to $p_2$ with $P_2^*\subseteq \eta(T_2)\setminus B(v_1)$. Also, for each $i\in \{1,3\}$, there exists a path $P_i$ in $G$ from $a$ to $b_i$ with $P_i^*\subseteq \eta(T_i)\setminus B(v_1)$.  Note that since $\eta$ is rich, it follows that $P$ is anticomplete to $N_G[a]$; in particular, $P_1$ has length at least two. But now $G$ contains a theta with ends $a$ and $b_1$ and paths $P_1, a\dd P_2\dd p_2\dd P\dd p_1\dd b_1$ and $b_1\dd b_3\dd P_3\dd a$, a contradiction. This proves \eqref{almost_one_end_complete}.\medskip

The following is immediate from \eqref{almost_one_end_complete} and the fact that $T$ is smooth.

\sta{\label{one_end_complete} Let $P$ be a path in $M$ with ends $p_1$ and $p_2$ such that there exists $x_1\in N_{\eta(T)}(p_1)$ and $x_2\in N_{\eta(T)}(p_2)$ for which $\{x_1,x_2\}$ is not local in $\eta$, and such that $|P|\geq 1$ is minimum subject to this property. Suppose that there exist $\{k_1,k_2\}=\{1,2\}$ and  $f=v_1v_2\in E(T)$ such that $x_{k_1}\in B(v_{k_1})\setminus (\eta(f))$ and $x_{k_2}\in (B(v_{k_2})\cup \eta(f))\setminus B(v_{k_1})$. Then $p_{k_1}$ is complete to $B(v_{k_1})\setminus \eta(f)$.}

We now deduce:

\sta{\label{localcomp} Let $D$ be a component of $M$. Then $N_{\eta(T)}(D)$ is local in $\eta$.}
     
Suppose not. By Lemma~\ref{twoattachments}, there exist $x_1,x_2 \in N_{\eta(T)}(D)$ such that $\{x_1,x_2\}$ is not local in $\eta$. For each $i\in \{1,2\}$, let $p_i$ be a neighbor of $x_i$ in $D$. Since $D$ is connected, there exists a path $P$ in $D\subseteq M$ from $p_1$ to $p_2$. In other words, there exists a path $P$ in $M$ with ends $p_1,p_2$ along with $x_1\in N_{\eta(T)}(p_1)$ and $x_2\in N_{\eta(T)}(p_2)$ such that $\{x_1,x_2\}$ is not local in $\eta$. Now,  let $P$ be a path in $M$ with ends $p_1$ and $p_2$ such that there exists $x_1\in N_{\eta(T)}(p_1)$ and $x_2\in N_{\eta(T)}(p_2)$ for which $\{x_1,x_2\}$ is not local in $\eta$, and such that $|P|\geq 1$ is minimum subject to this property. So we can apply \eqref{wideclaimpath} to $P$, $x_1$ and $x_2$. Let $\{j_1,j_2\}=\{1,2\}$ and $f=v_1v_2\in E(T)$ be as in \eqref{wideclaimpath}. We may assume without loss of generality that $j_1=1$ and $j_2=2$; thus, $v_1$ is a branch vertex of $T$. It follows from \eqref{edgeexists}  that $N_{\eta(T)}(P^*)\subseteq \eta(f,v_{1})$ and $N_{\eta(T)}(\{p_1,p_2\})\subseteq \eta(f)\cup B(v_1)\cup B(v_2)$. By \eqref{one_end_complete} applied to $k_1=1$ and $k_2=2$, $p_1$ is complete to $B(v_1)\setminus \eta(f)$. Also, from \eqref{one_end_complete} applied to $k_1=2$ and $k_2=1$, it follows that either $p_2$ is complete to $B(v_2)\setminus \eta(f)$ and $B(v_2)\setminus \eta(f)\neq \emptyset$, or $p_2$ is anticomplete to $B(v_2)\setminus \eta(f)$. Note that if $|P|>1$, then by the minimality of $|P|$, we have $N_{\eta(T)}(p_1)\subseteq B(v_1)$ and $N_{\eta(T)}(p_2)\subseteq (B(v_2)\cup \eta(f))\setminus B(v_1)$. Let us define  $\eta': V(T)\cup E(T)\cup (E(T)\times V(T))\subseteq 2^{G\setminus \{a\}}$ as follows. Let $\eta'(f)=\eta(f)\cup P$ and  let $\eta'(f,v_1)=\eta(f,v_1)\cup \{p_1\}$. Let
\begin{itemize}
    \item $\eta'(f,v_2)=\eta(f,v_2)\cup \{p_2\}$ if $p_2$ is complete to $B(v_2)\setminus \eta(f)$ and $B(v_2)\setminus \eta(f)\neq \emptyset$; and
     \item $\eta'(f,v_2)=\eta(f,v_2)$ if $p_2$ is anticomplete to $B(v_2)\setminus \eta(f)$.
\end{itemize}
Let $\eta'=\eta$ elsewhere on $V(T)\cup E(T)\cup (E(T)\times V(T))$. Then since $\eta$ is tame, substantial and rich, and $p_2$ is adjacent to $x_2\in B(v_2)\cup \eta(f))\setminus B(v_1)$, it is straightforward to check that $\eta'$ is also a tame, substantial and rich $(T,a)$-strip-structure. But we have $\eta'(T)=\eta(T)\cup P$, a contradiction with the maximality of $\eta(T)$. This proves \eqref{localcomp}.\medskip

The proof is almost concluded. Let $X$ be the union of all the components $D$ of $M$ such that $D$ is anticomplete to $\eta^+(T)$. Since $\eta$ is rich, it follows that $a$ is anticomplete to $M\setminus X$, as well. Thus, for every component $D$ of $M\setminus X$, $N_{\eta^+(T)}(D)=N_{\eta(T)}(D)$ is non-empty.
By \eqref{localcomp}, for every component $D$ of $M \setminus X$, $N_{\eta(T)}(D)$ is local in $\eta$. Let $\mathcal{D}$ be the set of all components $D$ of $M\setminus X$ for which we have $N_{\eta^+(T)}(D) \subseteq B_{\eta}(v)$ for some $v\in V(T)$. Breaking the ties arbitrarily and by the definition of $X$, we may write $\mathcal{D}=\bigcup_{v\in V(T)}\mathcal{D}_v$, where
\begin{itemize}
    \item for all distinct $u,v\in V(T)$, we have $\mathcal{D}_u\cap\mathcal{D}_v=\emptyset$; and
    \item for all $v\in V(T)$ and every $D\in \mathcal{D}_v$, we have $N_{\eta^+(T)}(D) \subseteq B_{\eta}(v)$ and $N_{\eta^+(T)}(D)\neq \emptyset$.
\end{itemize}
Also, for every $e=uv\in E(T)$, let $\mathcal{D}_e$ be the set of all components $D$ of $M\setminus X$ for which we have $N_{\eta^+(T)}(D)\subseteq \eta(e)$ and 
\begin{itemize}
    \item either $N_{\eta(T)}(D)\cap \eta^{\circ}(e)\neq \emptyset$, or;
    \item $N_{\eta(T)}(D)\cap (\eta(e,u)\setminus \eta(e,v))\neq \emptyset$ and $N_{\eta(T)}(D)\cap (\eta(e,v)\setminus \eta(e,v))\neq \emptyset$.
\end{itemize}
From the definition of $X$, it follows that every component of $M\setminus X$ belongs to exactly one of the sets $\{\mathcal{D}_v,\mathcal{D}_e:v\in V(T),e\in E(T)\}$ (note that since $\eta$ is rich, $a$ is anticomplete to each such component).

Let $\zeta: V(T)\cup E(T)\cup (E(T)\times V(T))\subseteq 2^{G\setminus \{a\}}$ be defined as follows. For all $v\in V(T)$ and $e\in E(T)$, let
\begin{itemize}
    \item $\zeta(v)=\bigcup_{D\in \mathcal{D}_v}D$; 
    \item $\zeta(e)=\eta(e)\cup (\bigcup_{D\in \mathcal{D}_e}D)$; and
    \item $\zeta(e,v)=\eta(e,v)$.
\end{itemize}
It is easily seen that $\zeta$ satisfies the conditions (S\ref{s1}-S\ref{s8}) from the definition of a $(T,a)$-strip-structure. In particular, since $\eta$ is rich, $\zeta$ satisfies (S\ref{s2}), and from the definitions of $X$, $\mathcal{D}_v$ and $\mathcal{D}_e$, it follows that $\zeta$ satisfies (S\ref{s5}) and (S\ref{s7}). Also, we have $\eta\leq \zeta$.
  
Now, since $\eta$ is substantial and rich, since $\eta\leq \zeta$, and from the definitions of $X$ and $\zeta$, it follows that $\zeta$ is a substantial and rich $(T,a)$-strip-structure with $\mathcal{J}_{\zeta}=\mathcal{J}_{\eta}$. Moreover, note that we have $\zeta^+(T)=\eta(T)^+\cup (M\setminus X)$, and so $G\setminus (\zeta^+(T)\cup \mathcal{J}_{\zeta})=G\setminus (\zeta^+(T)\cup \mathcal{J}_{\eta})=X$ is anticomplete to $\zeta^+(T)$. This completes the proof of Theorem~\ref{stripstructure}.  
\end{proof}

\section{Jewels under the loupe}\label{sec:jewelsconnectivity}
Here we revisit jewels for strip-structures, establishing several results about their properties in various settings. This will help attune Theorem~\ref{stripstructure} for its application in the proof of Theorem~\ref{mainjewelconnectivity}.

First we need to introduce some notations. Let $G$ be a graph and let $a\in G$. Let $T$ be a smooth tree and let $\zeta $ be a $(T,a)$-strip-structure in $G$. Let $v\in V(T)$ and let $e\in E(T)$ be incident with $v$. We denote by $\zeta_e(v)$ the set of all components $D$ of $\zeta(v)$ for which we have $N_{B(v)}(D)\subseteq \zeta(e,v)$, or equivalently, $N_{\zeta(T)\setminus \zeta(e,v)}(D)=\emptyset$.

Let $(v,e_1,e_2)$ be a seagull in $T$ and let $u_i$ be the end of $e_i$ distinct from $v$ for each $i\in \{1,2\}$. We define
$$\zeta(v,e_1,e_2)=\zeta(e_1)\cup \zeta(e_2)\cup \zeta_{e_1}(u_1)\cup \zeta_{e_2}(u_2)\cup \zeta(v).$$
We denote by $\mathcal{J}_{\zeta,(v,e_1,e_2)}$ the set of all jewels for $\zeta$ at $(v,e_1,e_2)$, and for every vertex $v\in V(T)$, $\mathcal{J}_{\zeta,v}$ stands for the set of all jewels for $\zeta$ at $v$. It follows that $\mathcal{J}_{\zeta,v}=\emptyset$ if $v$ is a leaf of $T$. 

The first result in this section describes, for a $(T,a)$-strip-structure in a theta-free graph, the attachments of jewels at a vertex of $T$.

\begin{theorem}\label{wherejewelsattach}
  Let $G$ be a theta-free graph and let $a\in V(G)$. Let $T$ be a smooth tree and let $\zeta$ be a $(T,a)$-strip-structure in $G$. Let $(v,e_1,e_2)$ be a seagull in $T$ and let $x \in \mathcal{J}_{\zeta,(v,e_1,e_2)}$. Then the following hold.
  \begin{itemize}
    \item We have $N_{\zeta^+(T)}(x)\subseteq \zeta(v,e_1,e_2)$, and so $N_{\zeta^+(T)}(\mathcal{J}_{\zeta,(v,e_1,e_2)})\subseteq \zeta(v,e_1,e_2)$. Consequently, for every vertex $v\in V(T)$, we have $N_{\zeta^+(T)}(\mathcal{J}_{\zeta,v})\subseteq \zeta(N_T[v])$, and for every two distinct vertices $v,v'\in V(T)$, we have $\mathcal{J}_{\zeta,v}\cap \mathcal{J}_{\zeta,v'}=\emptyset$.
    \item Assume that $\zeta$ is rich. Let $i \in \{1,2\}$ and let $R$ be a long $\zeta(e_i)$-rung, let $r$ be the end of $R$ in $\zeta(e_i,v)$ and let $r'$ be the unique neighbor of $r$ in $R$. Then either $x$ is anticomplete to $R$ or $N_R(x)=\{r,r'\}$.
\end{itemize}
\end{theorem}

\begin{proof}
 Note that $v$ is a branch vertex of $T$. For each $i\in \{1,2\}$, let $u_i$ be the end of $e_i$ distinct from $v$ and let $T_i$ be the component of $T\setminus (N_T(v)\setminus \{u_i\})$ containing $v$. Let $T'$ be the component of $T\setminus \{u_1,u_2\}$ containing $v$. Let $x\in \mathcal{J}_{\zeta,(v,e_1,e_2)}$. Since $x \in \mathcal{J}_{\zeta,(v,e_1,e_2)}$ is a jewel for $\zeta$, there exist an edge $e_3 \in E(T)\setminus \{e_1,e_2\}$ incident with $v$ and  a $\zeta$-pyramid $\Sigma$ at $(v,e_1,e_2,e_3)$ with apex $a$, base $b_1b_2b_3$, and paths $P_1,P_2,P_3$ such that $x$ is a jewel for $\Sigma$ at $b_3$. In particular, for each $j\in \{1,2,3\}$, $P_j\cap \zeta(e_j)$ is a long $\zeta(e_j)$-rung $R_j$ with $b_j$ as its end in $\zeta(e_j,v)$. Also, $x$ is anticomplete to $P_3$ (and so $x$ is not adjacent to $a$), and for each $j\in \{1,2\}$, assuming $c_j$ to be the unique vertex in $N_{R_j}(b_j)=N_{P_j}(b_j)$, $x$ is adjacent to $b_j\in \zeta(e_j,v)\setminus \zeta(e_j,u_j)$ and $c_j\in \zeta(e_j)\setminus \zeta(e_j,v)$. Therefore, there exist paths $Q_i,S_i$ of length more than one in $G$ from $a$ to $x$ for which we have $b_i\in Q_i^*\subseteq (\zeta(T')\setminus \zeta(v))\cup (\zeta(e_i,v)\setminus \zeta(e_i,u_i))$ and $c_i\in S_i^*\subseteq \zeta(T_i)\setminus (B(v)\cup \zeta(u_i)\cup \zeta(v))$.

To prove the first assertion of Theorem~\ref{wherejewelsattach}, assume for a contradiction that $x$ has a neighbor $y\in \zeta^+(T)\setminus \zeta(v,e_1,e_2)$. Since $x$ is not adjacent to $a$, we have $y\in \zeta(T)\setminus \zeta(v,e_1,e_2)$. First, assume that $y\in \zeta(T')\setminus \zeta(v)$. Then by (S\ref{s5}) and (S\ref{s7}) from the definition of a strip-structure, there exists a path $Q'$ of length more than one in $G$ from $a$ to $x$ with $Q'^*\subseteq \zeta(T')\setminus \zeta(v)$. But now $G$ contains a theta with ends $a,x$ and paths $a\dd S_1\dd x$, $a\dd S_2\dd x$ and $a\dd Q'\dd x$, a contradiction. It follows that $y\in \zeta(T_1\cup T_2)\setminus \zeta(v,e_1,e_2)$. In other words, for some $i\in \{1,2\}$, we have $y\in \zeta(T_i)\setminus (\zeta(e_i)\cup \zeta_{e_i}(u_i)\cup \zeta(v))$. As a result, by (S\ref{s5}) and (S\ref{s7}) from the definition a strip-structure, and by the definition of $\zeta_{e_i}(u_i)$, there exists a path $S_i'$ of length more than one in $G$ from $a$ to $x$ with $S_i'^*\subseteq \zeta(T_i)\setminus (\zeta(e_i)\cup \zeta_{e_i}(u_i)\cup \zeta(v))$. But now assuming $i'\in \{1,2\}$ to be distinct from $i$, $G$ contains a theta with ends $a,x$ and paths $a\dd Q_i\dd x$, $a\dd S_i'\dd x$ and $a\dd S_{i'}\dd x$, a contradiction. This proves the the first assertion.

Next we prove the second assertion of Theorem~\ref{wherejewelsattach}.  By symmetry, we may assume that $i=1$. Assume that $x$ has a neighbor $y\in R$. Let $P_1'=(P_1\setminus R_1)\cup R$. Let $\Sigma'$ be the pyramid with apex $a$, base $rb_2b_3$, and paths $P_1', P_2$ and $P_3$. Recall that since $\zeta$ is rich, $a$ is trapped in $\zeta^+(T)$. Also, $\Sigma'$ is a pyramid in $\zeta^+(T)$, $x$ is adjacent to $y\in P_1'$, $x$ is adjacent to $b_2,c_2\in P_2$, and $x$ is anticomplete to $P_3$. It follows that $x$ is a wide vertex for $\Sigma'$ which is not a corner path for $\Sigma'$. Now applying Lemma~\ref{major} to $G$, $a$, $H=\zeta^+(T)$, $\Sigma'$ and $p=x$, we deduce that $x$ is a jewel for $\Sigma'$ at $b_3$, and so $N_{R}(x)=N_{P_1'}(x)=\{r,r'\}$. This completes the proof of Theorem~\ref{wherejewelsattach}.
\end{proof}

Our next goal is to show that for every rich $(T,a)$-strip-structure in a graph $G\in \mathcal{C}_t$, there are only a few jewels at each vertex of $T$. Let us begin with a lemma, asserting that for a rich $(T,a)$-strip-structure $\zeta$ in a theta-free graph, each set $B_{\zeta}(v)$ is almost a clique.
\begin{lemma}\label{B(v)almostclique}
        Let $G$ be a theta-free graph and $a \in V(G)$. Let $T$ be a smooth tree and $\zeta$ be a rich $(T,a)$-strip-structure in $G$. Then for every $v\in V(T)$, there exists at most one edge $f\in E(T)$ such that $\zeta(f,v)$ is not a clique.
\end{lemma}
\begin{proof}
Suppose for a contradiction that there are two distinct edges $f_1,f_2\in E(T)$ incident with $v$, and for each $i\in \{1,2\}$, there exist $x_i,y_i\in \zeta(f_i,v)$ such that $x_i$ is not adjacent to $y_i$. Then $v$ is not a leaf of $T$ and $H=x_1\dd x_2\dd y_1\dd y_2\dd x_1$ is a hole of length four in $G$. Since $\zeta$ is rich, $a$ is anticomplete to $H$. Let $f_1=u_1v$. Let $l_1$ be a leaf of $T$ which belongs to the component of $T\setminus \{v\}$ containing $u_1$, and let $\Lambda_1$ be the unique path in $T$ from $v$ to $l_1$ (so $f_1\in E(\Lambda_1)$). Let $R_{x_1}$ be an $\zeta(f)$-rung containing $x_1$ and let $R_{y_1}$ be an $\zeta(f)$-rung containing $y_1$. Since $\zeta$ is rich, $H_1=R_{x_1}\cup R_{x_2}\cup B(u_1)$ is a connected induced subgraph of $G$, and so there is a path $Q$ in $H_1$ from $x_1$ to $y_1$. It follows that $Q$ has length more than one and $Q^*\subseteq (B(u_1)\cup \zeta(f_1))\setminus B(v)$. But now $G$ contains a theta with ends $x_1,y_1$ and paths $Q$, $x_1\dd x_2\dd y_1$ and $x_1\dd y_2\dd y_1$, a contradiction. This completes the proof of Lemma~\ref{B(v)almostclique}.
\end{proof}
Recall the following classical result of Ramsey (see, for instance, \cite{ajtai} for an explicit bound.)
\begin{theorem}[See \cite{ajtai}]\label{classicalramsey}
For all integers $a,b\geq 1$, there exists an integer $R=R(a, b)\geq 1$ such that every graph $G$ on at least $R(a,b)$ vertices contains either a clique of cardinality $a$ or a stable set of cardinality $b$.
\end{theorem}

We can now prove the second main result of this section.

\begin{theorem}\label{fewjewels}
For all positive integers $t,\delta$, there exists a positive integer $j=j(t,\delta)$ with the following property. Let $G\in \mathcal{C}_t$ be a graph, let $a\in G$ and let  $T$ be a smooth tree of maximum degree $\delta$. Let  $\zeta$ be a rich $(T,a)$-strip-structure in $G$. Then for every vertex $v\in V(T)$, we have $|\mathcal{J}_{\zeta,v}|<j$.
\end{theorem}
\begin{proof}
  Let $ j=j(t,\delta)={\delta \choose 2}R(t,3)$ with $R(\cdot,\cdot)$ as in Theorem~\ref{classicalramsey}. Then in order to prove $|\mathcal{J}_{\zeta,v}|<j$, it is enough to show that $|\mathcal{J}_{\zeta,(v,e_1,e_2)}|<R(t,3)$ 
  for every seagull $(v,e_1,e_2)$ in $T$. Suppose for a contradiction that $|\mathcal{J}_{\zeta,(v,e_1,e_2)}| \geq R(t,3)$ for some seagull $(v,e_1,e_2)$ in $T$. Then $v$ is a branch vertex of $T$. For each $i \in \{1,2\}$, let $u_i$ be the end of $e_i$ different from $v$. Since $G \in \mathcal{C}_t$, it follows from Theorem~\ref{classicalramsey} that ${J}_{\zeta,(v,e_1,e_2)}$ contains a stable set $X$ of cardinality three. For every $x\in X$, since $x$ is a jewel for $\zeta$ at $(v,e_1,e_2)$, it follows that for every $i\in \{1,2\}$, there exists a long $\zeta(e_i)$-rung $R_i^x$ such that $Q_i^x=R_i^x\setminus \zeta(e_i,v)$ is a path in $\zeta(e_i)\setminus \zeta(e_i,v)$ from a neighbor of $x$ to a vertex in $\zeta(e_i,u_i)\setminus \zeta(e_i,v)$; in particular, $R_i^x$ contains a neighbor of $x$. Therefore, for each $i\in \{1,2\}$, we may pick a non-empty set $\mathcal{R}_i$ of long $\zeta(e_i)$-rungs such that every vertex in $X$ has a neighbor in at least one rung in $\mathcal{R}_i$, and with $\mathcal{R}_i$ minimal with respect to inclusion. We deduce:

  \sta{\label{bigminimal} There exists $i\in \{1,2\}$ with $|\mathcal{R}_i|>1$.}
  
  Suppose not. Then for every $i\in \{1,2\}$, there exists a long $\zeta(e_i)$-rung $S_i$ such that
  every vertex in $X$ has a neighbor in $S_i$. Let $s_i$ be the end of $S_i$ in $\zeta(e_i,v)$ and $s_i'$ be unique neighbor of $s_i$ in $S_i$. By the second assertion of Theorem~\ref{wherejewelsattach}, $X$ is complete to $\{s_1',s_2'\}$. But now $X\cup \{s_1',s_2'\}$ is a theta in $G$ with ends $s_1',s_2'$, a contradiction. This proves \eqref{bigminimal}.\medskip 

  By \eqref{bigminimal} and due to symmetry, we may assume that $|\mathcal{R}_1|>1$. This, together with the minimality of $\mathcal{R}_1$, implies that there exist distinct vertices $x,y\in X$ as well as distinct long $\zeta(e_1)$-rungs $R_x,R_y\in \mathcal{R}_1$ such that $x$ has a neighbor in $R_x$, $y$
  has a neighbor in $R_y$, $x$ is anticomplete to $R_y$, and $y$ anticomplete to
  $R_x$. Let $r_x$ and $r_y$ be the ends of $R_x$ and $R_y$ in $\zeta(e_1, v)$, respectively. Let $r'_x$ be the unique neighbor of $r_x$ in $R_x$ and $r'_y$ be the unique neighbor of $r_y$ in $R_y$. So we have $r_x',r_y'\in \zeta(e_1)\setminus \zeta(e_1, v)$. By the second assertion of Theorem~\ref{wherejewelsattach}, we have
  $N_{R_x \cup R_y}(x)=\{r_x,r_x'\}$ and $N_{R_x \cup R_y}(y)=\{r_y,r_y'\}$. It follows that $r_x,r_x'\in R_x\setminus R_y$ and $r_y,r_y'\in R_y\setminus R_x$. Also, $r_x$ is anticomplete to $R_y\setminus \{r_y\}$, as otherwise $(R_y\setminus \{r_y\})\cup \{r_x\}$ contains a long $\zeta(e_1)$-rung $R$ with $N_R(x)=\{r_x\}$, which violates the second assertion of Theorem~\ref{wherejewelsattach}. Similarly, $r_y$ is anticomplete to $R_x\setminus \{r_x\}$.
  
  Now, let $G_1=G[(B(u_1)\setminus \zeta(e_1,u_1))\cup ((R_x\cup R_y)\setminus \{r_x,r_y\})]$ and let $G_2=G[(B(u_2)\setminus \zeta(e_2,u_2))\cup Q_2^x\cup Q_2^y]$. Since $\zeta$ is rich, the second bullet in the definition of a rich strip-structure implies that $G_1$ and $G_2$ are connected. Consequently, there exists a path $Q_1$ in $G_1$ from $r'_x$ to $r'_y$, and 
  there exists a path $Q_2$ from $x$ to $y$ with $Q_2^*\subseteq G_2$. Also, since $v$ is a branch vertex of $T$,  we may choose an edge $e_3 \in E(T)\setminus \{e_1,e_2\}$ incident with $v$. By the first assertion of Theorem~\ref{wherejewelsattach}, $\{x,y\}$ is anticomplete to $\zeta(e_3,v)$.
  Let $Q_3$ be a path from $r_x$ to $r_y$ with $Q_3^*\subseteq \zeta(e_3,v)$
  (thus $|Q_3| \in \{2,3\}$). But now $G$ contains a prism with triangles $xr_xr'_x$ and $yr_yr'_y$ and paths
  $Q_1,Q_2,Q_3$, a contradiction. This completes the proof of Theorem~\ref{fewjewels}.
\end{proof}

Our last theorem in this section examines the connectivity within $G\setminus \zeta^+(T)$ for a $(T,a)$-strip-structure $\zeta$ arising from Theorem~\ref{stripstructure}. We need the following lemma, the proof of which is similar to that of Theorem~\ref{wherejewelsattach}.

\begin{lemma}\label{distantjewels}
Let $G$ be a theta-free graph and let $a\in V(G)$. Let $T$ be a smooth tree and let $\zeta $ be a $(T,a)$-strip-structure in $G$. Let $v,v'\in V(T)$ be distinct and let $P$ be a path in $G\setminus \zeta^+(T)$ with ends $x,x'$ such that $x\in \mathcal{J}_{\zeta,v}$, $x'\in \mathcal{J}_{\zeta,v'}$, and $P^*$ is anticomplete to $\zeta^+(T)$. Then $v$ and $v'$ are adjacent in $T$.
\end{lemma}
\begin{proof}
    Suppose not. Note that by Theorem~\ref{wherejewelsattach}, $x$ and $x'$ are distinct. Let $\Lambda$ be the path in $T$ from $v$ to $v'$.  Then $\Lambda$ has length more than one, and so there are two  distinct edges $f,f'\in E(\Lambda)$ such that $f$ is incident with $v$ and $f'$ is incident with $v'$. Let $u$ be the end of $f$ distinct from $v$ and $u'$ be the end of $f'$ distinct from $v'$. Let $(v,e_1,e_2)$ and $(v',e_1',e_2')$ be two seagulls in $G$ such that $x\in \mathcal{J}_{\zeta,(v,e_1,e_2)}$ and $x'\in \mathcal{J}_{\zeta,(v',e_1',e_2')}$. For each $i\in \{1,2\}$, let $u_i$ be the end of $e_i$ distinct from $v$ and let $u_i'$ be the end of $e_i'$ distinct from $v'$. Without loss of generality, we may assume that $u_2,u_2'\notin \Lambda$. Let $T_2$ be the component of $T\setminus (N_T(v)\setminus \{u_2\})$ containing $v$ and let $T_2'$ be the component of $T\setminus (N_T(v')\setminus \{u_2'\})$ containing $v'$. Let $T'$ be the component of $T\setminus \{u',u_2'\}$ containing $v'$.  Since $x$ is a jewel for $\zeta$ at $(v,e_1,e_2)$, it follows that $x$ is not adjacent to $a$, and $x$ has a neighbor $c\in \zeta(e_2)\setminus \zeta(e_2,v)\subseteq \zeta(T_2)\setminus (B(v)\cup \zeta(u_2)\cup \zeta(v))$. Therefore, there exists a path $Q$ of length more than one in $G$ from $a$ to $x$ for which we have $c\in Q^*\subseteq \zeta(T_2)\setminus (B(v)\cup \zeta(u_2)\cup \zeta(v))$. Also, since $x'$ is a jewel for $\zeta$ at $(v',e'_1,e'_2)$, it follows that $x'$ is not adjacent to $a$, and $x'$ has a neighbor $b'\in B(v')\setminus (\zeta(f',u')\cup \zeta(e_2',v'))$ and a neighbor $c'\in \zeta(e_2')\setminus \zeta(e_2',v')\subseteq \zeta(T_2')\setminus (B(v')\cup \zeta(u_2')\cup \zeta(v'))$. Therefore, there exist paths $P',Q'$ of length more than one in $G$ from $a$ to $x'$ for which we have $b'\in P'^*\subseteq (\zeta(T')\setminus \zeta(v'))\cup (\zeta(f',v')\setminus \zeta(f',u'))$ and $c'\in Q'^*\subseteq \zeta(T_2')\setminus (B(v')\cup \zeta(u_2)\cup \zeta(v'))$. But now $G$ contains a theta with ends $a,x'$ and paths $a\dd P'\dd x'$, $a\dd Q'\dd x'$, and $a\dd Q\dd x\dd P\dd x'$, a contradiction. This proves Lemma~\ref{distantjewels}.
\end{proof}

\begin{theorem}\label{antijewelseparation}
Let $t,\delta\geq 1$ be integers and let $j(t,\delta)$ be as in Theorem~\ref{fewjewels}. Let $G\in \mathcal{C}_t$ be a graph and let $a\in V(G)$. Let $T$ be a smooth tree of maximum degree $\delta$ and let $v\in V(T)$. Let $\zeta$ be a rich $(T,a)$-strip-structure in $G$ such that $G\setminus (\zeta^+(T)\cup \mathcal{J}_{\zeta})$ is anticomplete to $\zeta^+(T)$. Let $x\in G\setminus (\zeta^+(T)\cup \mathcal{J}_{\zeta})$. Then there exists $S_x\subseteq G\setminus (\zeta^+(T)\cup \{x\})$ such that $|S_x|< 2j(t,\delta)$ and $S_x$ separates $x$ and $\mathcal{J}_{\zeta}\setminus (\{x\}\cup S_x)$ in $G\setminus \zeta^+(T)$. Consequently, $S_x$ separates $x$ and $\zeta^+(T)$ in $G$.
\end{theorem}
\begin{proof}
By Theorem~\ref{wherejewelsattach}, $\{\mathcal{J}_{\zeta,v}:v\in V(T)\}$ is a partition of $\mathcal{J}_{\zeta}$. Let $G'$ be the graph obtained from $G\setminus \zeta^+(T)$ by contracting the set $\mathcal{J}_{\zeta,v}$ into a vertex $z_{v}$ for each $v\in V(T)$ with $\mathcal{J}_{\zeta,v}\neq \emptyset$,  and then adding a new vertex $z$ such that $N_{G'}(z)=\{z_{v}: v\in V(T), \mathcal{J}_{\zeta,v}\neq \emptyset\}$. We claim that there is a set $Y\subseteq G'\setminus \{x,z\}$ of cardinality at most two which separates $x$ and $z$ in $G'$. Suppose not. By Theorem~\ref{Menger}, there are three pairwise internally disjoint paths in $G'$ from $x$ to $z$. Thus, there exist $S\subseteq T$ with $|S|=3$ as well as three paths $\{P_{v}:v\in S\}$ in $G\setminus \zeta^+(T)$ all having $x$ as an end and otherwise disjoint, such that for each $v\in S$, $P_{v}$ has an end $y_{v}\in \mathcal{J}_{\zeta,v}$ distinct from $x$, and we have $P_{v}^*\subseteq G\setminus (\zeta^+(T)\cup \mathcal{J}_{\zeta})$. As a result, for all distinct $v,v'\in S$, $P_{v,v'}=y_{v}\dd P_{v}\dd x\dd P_{v'}\dd y_{v'}$ is a path in $G\setminus \zeta^+(T)$ from $y_{v}\in \mathcal{J}_{\zeta,v}$ to $y_{v'}\in \mathcal{J}_{\zeta,v'}$ such that $P_{v,v'}^*\subseteq G\setminus (\zeta^+(T)\cup \mathcal{J}_{\zeta})$. In particular, $P_{v,v'}^*$ is anticomplete to $\zeta^+(T)$. But then by Lemma~\ref{distantjewels}, $S$ is a clique in $T$, which is impossible. The claim follows.

Let $Y$ be as in the above claim. For each $y\in Y$, if $y=z_{v}$ for some $v\in V(T)$, then let $A_y=\mathcal{J}_{\zeta,v}$. Otherwise, let $A_y=\{y\}$. Let $S_x=\bigcup_{y\in Y}A_y$. Then $S_x\subseteq G\setminus (\zeta^+(T)\cup \{x\})$ separates $x$ and $\mathcal{J}_{\zeta}\setminus (\{x\}\cup S_x)$ in $G\setminus \zeta^+(T)$.  Also, by Theorem~\ref{fewjewels}, we have $|S_x|< 2j(t,\delta)$. This completes the proof of Theorem~\ref{antijewelseparation}.
\end{proof}

\section{Strip structures and connectivity}\label{sec:stripconnectivity}
In this section, we investigate the connectivity implications of the presence of certain $(T,a)$-strip-structures in graphs from $\mathcal{C}_t$. The main result is the following.
\begin{theorem}\label{mainjewelconnectivity}
    For all integers $t,\delta\geq 1$, there exists an integer $\sigma=\sigma(t,\delta)\geq 1$ with the following property. Let $G\in \mathcal{C}_t$ be a graph and let $a\in V(G)$. Let $T$ be a smooth tree of maximum degree $\delta$ and let $\zeta $ be a rich $(T,a)$-strip-structure in $G$ such that $G\setminus (\zeta^+(T)\cup \mathcal{J}_{\zeta})$ is anticomplete to $\zeta^+(T)$. Then for every vertex $x\in G\setminus N_G[a]$, there exists a set $S_x\subseteq G\setminus \{a,x\}$ with $|S_x|<\sigma$ such that $S$ separates $a$ and $x$ in $G$.
    \end{theorem}
\begin{proof}
   Let $j(t,\delta)$ be as in Theorem~\ref{fewjewels}. We claim that 
   $$\sigma=\sigma(t,\delta)=2\delta (j(t,\delta)+ t)$$
   satisfies Theorem~\ref{mainjewelconnectivity}. For every vertex $v\in V(T)$, we define $C_v=B(v)$ if $v$ is a leaf of $T$  and $C_v=\emptyset$ otherwise. Also, for every vertex $v\in V(T)$, let $K_v$ be a maximal clique of $G$ contained in $B(v)$. Thus, we have $|K_v|<t$. Moreover, Lemma~\ref{B(v)almostclique} along with the assumption that $\zeta$ is rich implies that if $v$ is a leaf of $T$, then we have $K_v=B(v)=C_v$  (and so $|K_v|=1$), and if $v$ is a branch vertex of $T$, then $K_v$ contains all but possibly one of the sets $\zeta(f,v)$ for $f\in E(T)$. For every $S\subseteq T$, we define
   $$\mathcal{M}_S=\bigcup_{w\in N_T(S)}\mathcal{J}_{\zeta,w},$$
$$\mathcal{N}_S=\bigcup_{w\in N_T(S)}K_w.$$
Also, we write $\mathcal{M}_v$ for $\mathcal{M}_{\{v\}}$ and $\mathcal{N}_v$ for $\mathcal{N}_{\{v\}}$. For every $v\in V(T)$, let $\mathcal{O}_v=\mathcal{M}_v\cup \mathcal{N}_v$.  The following is immediate from Theorems~\ref{wherejewelsattach} and \ref{fewjewels} and Lemma~\ref{distantjewels}.

\sta{\label{jewelsep}For every $v\in V(T)$, we have
\begin{itemize}
    \item $\mathcal{O}_v\subseteq G\setminus (\mathcal{J}_{\zeta,v}\cup \{a\})$;
    \item $|\mathcal{O}_v|< \delta(j(t,\delta)+t)\leq \sigma$; and
    \item $\mathcal{O}_v$ separates $a$ and $\mathcal{J}_{\zeta,v}$ in $G$.
\end{itemize}}
   
 Now, for every $x\in G\setminus N_G[a]$, we define $S_x$ as follows. First, assume that $x\in \zeta(T)\setminus N_G[a]$. Then either $x\in \zeta(e)$ for some edge $e=uv\in E(T)$, or $x\in \zeta(v)$ for some branch vertex $v\in V(T)$. In the former case, let
$$\mathcal{E}_x=\mathcal{M}_u\cup \mathcal{M}_v,$$
$$\mathcal{I}_x=\mathcal{N}_{\{u,v\}}\cup C_u\cup C_v.$$
In the latter case, let
$$\mathcal{E}_x=\mathcal{M}_v\cup \mathcal{J}_{\zeta,v}$$
$$\mathcal{I}_x=\mathcal{N}_v.$$
Let $S_x=\mathcal{E}_x\cup \mathcal{I}_x$. Observe that since $x\in G\setminus N_G[a]$, we have $S_x\subseteq G\setminus \{a,x\}$. Also,  by Theorem~\ref{fewjewels}, we have $|\mathcal{E}_x|\leq 2\delta j(t,\delta)$ and so $|S_x|<2\delta (j(t,\delta)+ t)=\sigma$. Moreover, from Theorem~\ref{wherejewelsattach} and the fact that $\zeta$ is rich, it is easy to check that for every path $P$ in $G$ from $a$ to $x$, if $P\subseteq \zeta^+(T)$, then $P$ contains a vertex from $\mathcal{I}_x$, and otherwise $P$ contains a vertex from either $\mathcal{I}_x$ or $\mathcal{E}_x$. Therefore, $S_x$ separates $a$ and $x$ in $G$.

Next, assume that $x\in \mathcal{J}_{\zeta}$. Then  by Theorem~\ref{wherejewelsattach}, there exists a unique vertex $v\in V(T)$ such that $x\in \mathcal{J}_{\zeta,v}$. Let $S_x=\mathcal{O}_v$. Then by \eqref{jewelsep}, we have $S_x\subseteq G\setminus \{a,x\}$, $|S_x|<\sigma$ and $S_x$ separates $a$ and $x$ in $G$. 

Finally, assume that $x\in G\setminus (\zeta^+(T)\cup \mathcal{J}_{\zeta})$. Then letting $S_x$ to be as in Theorem~\ref{antijewelseparation}, it follows from Theorem~\ref{antijewelseparation} that $S_x\subseteq G\setminus \{a,x\}$, $|X|< 2j(t,\delta)\leq \sigma$ and $S_x$ separates $a$ and $x$ in $G$. This completes the proof of Theorem~\ref{mainjewelconnectivity}.
\end{proof}

Our application of Theorem~\ref{mainjewelconnectivity} is confined to the case where $T$ is a caterpillar. More precisely, for a graph $G$ and a vertex $a \in G$, an induced subgraph $H \subseteq G\setminus \{a\}$ is said to be an
{\em $a$-seed} in $G$ if the following hold.
\begin{itemize}
    \item There exists a caterpillar $C$ such that $H$ is the line graph of a $1$-subdivision of $C$ and $N_G(a)=\mathcal{Z}(H)$.
    \item The vertex $a$ is trapped in $H\cup \{a\}$.
\end{itemize}
It follows that $\mathcal{Z}(H)$ is the set of all degree-one vertices of $H$. We now combine Theorems~\ref{stripstructure} and \ref{mainjewelconnectivity} to deduce the following.
\begin{theorem} \label{noseed}
  For every integer $t\geq 1$, there exists an integer $s=s(t)\geq 1$ with the following property. Let $G \in \mathcal{C}_t$ be a graph and let
  $a \in V(G)$. Assume that there is an $a$-seed in $G$. Then for every vertex $x\in G \setminus N_G[a]$, there exists $S_x \subseteq G\setminus \{a,x\}$
  with $|S_x|< s$ such that $S_x$ separates $a$ and $x$ in $G$.
  \end{theorem}

\begin{proof}
Let $\sigma(\cdot,\cdot)$ be as in Theorem~\ref{mainjewelconnectivity}. We show that $s=s(t)=\sigma(t,3)$ satisfies Theorem~\ref{noseed}. Pick an $a$-seed $H$ in $G$. Let $T$ be the unique smooth caterpillar with $|N_G(a)|$ leaves. Then $T$ has maximum degree three. Also, one may immediately observe that there is a tame, substantial, and rich $(T,a)$-strip-structure $\eta$ in $G$ with $\eta(T)=H$. Now we can apply Theorem~\ref{stripstructure} to $G$, $a$, and $T$, deducing that there exists a substantial and rich $(T,a)$-strip-structure $\zeta$ in $G$ such that $G\setminus (\zeta^+(T)\cup \mathcal{J}_{\zeta})$ is anticomplete to $\zeta^+(T)$. Hence, by Theorem~\ref{mainjewelconnectivity} applied to $G$, $a$, $T$, and $\zeta$, for every vertex $x\in G\setminus N_G[a]$, there exists $S_x\subseteq G\setminus \{a,x\}$ with $|S_x|<s$ such that $S_x$ separates $a$ and $x$ in $G$. This completes the proof of Theorem~\ref{noseed}.
  \end{proof}

\section{From blocks to trees}
  \label{getatree}

In this section, we prove Theorem~\ref{mainthm}. We begin with a result which captures the use of Theorem~\ref{noseed} in the proof of  Theorem~\ref{mainthm}. For a positive integer $n$, we write $[n]=\{1,\ldots, n\}$. 
\begin{theorem}\label{banana}
    For all integers $t,\nu\geq 1$, there exists an integer $\psi=\psi(t,\nu)\geq 1$ with the following property. Let $G\in \mathcal{C}_t$, let $a,b\in V(G)$ be distinct and non-adjacent, and let $\mathcal{P}$ be a collection of pairwise internally disjoint paths in $G$ from $a$ to $b$ with $|\mathcal{P}|\geq \psi$. For each $P\in \mathcal{P}$, let $a_{P}$ be the neighbor of $a$ in $P$ (so $a_P\neq b$). Then there exists $P_1,\ldots, P_{\nu}\in \mathcal{P}$ such that:
    \begin{itemize}
    \item $\{a_{P_1},\ldots, a_{P_{\nu}},b\}$ is a stable set in $G$; and
    \item for all $i,j\in [\nu]$ with $i<j$, $a_{P_i}$ has a neighbor in $P_j^*\setminus \{a_{P_j}\}$.
    \end{itemize}
\end{theorem}
\begin{proof}
Let $s=s(t)$ be as in Theorem~\ref{noseed} and let $\mu=\mu(\max\{2s+1,t\})$, where $\mu(\cdot)$ is as in Theorem~\ref{connectifier}. Let $R(\cdot,\cdot)$ be as in Theorem~\ref{classicalramsey}. For every integer $p\geq 1$,  let $R_{tourn}(p)$ be the smallest positive integer $n$ such that every tournament on at least $n$ vertices contains a transitive tournament on $p$ vertices; the existence of $R_{tourn}(p)$ follows easily from Theorem~\ref{classicalramsey} (in fact, one may observe that $R_{tourn}(p)\leq R(p,p)$). Let $\gamma=R(R_{tourn}(\nu+1),\mu)$. We prove that 
$$\psi=\psi(t,\nu)=R(\gamma,t)$$
 satisfies Theorem~\ref{banana}. Let us choose $\psi$ distinct paths $P_1,\ldots, P_{\psi}\in \mathcal{P}$, and for each $i\in [\nu]$, let us write $a_i=a_{P_i}$. Since $G$ is $K_t$-free, it follows from Theorem~\ref{classicalramsey} and the definition of $\psi$ that there exists a stable set $N\subseteq \{a_i:i\in [\psi]\}$ in $G$ with $|N|=\gamma$; we may assume without loss of generality that $N=\{a_i:i\in [\gamma]\}$.

Let $D$ be a directed graph with $V(D)=N$ such that for distinct $i,j\in [\gamma]$, there is an arc from $a_i$
  to $a_j$ in $D$ if and only if $x_i$ has a neighbor in $P_j^*\setminus \{a_j\}$. Note that $D$ may contain both arcs $(a_i,a_j)$ and $(a_j,a_i)$, and so the undirected underlying graph of $D$ might not be simple. Let $D^-$ be the simple graph obtained from the undirected underlying graph of $D$ by removing one of every two parallel edges.

  \sta{\label{nostableset}$D^-$ contains no stable set of cardinality $\mu$.}
  Suppose for a contradiction that $D^-$ contains a stable set $S$ of cardinality $\mu$. We may
  assume without loss of generality that $S=\{a_1, \dots, a_{\mu}\}$.
  Let $G_1=G[(\bigcup_{j=1}^{\mu} P_j)\setminus \{a\}]$. Note that by the definition of $D$, for every $i \in [\mu]$, we have $N_{G_1}(a_i)=N_{P_i}(a_i) \setminus \{a\}$, and in  particular $|N_{G_1}(a_i)|=1$. Since $G_1$ is connected and $K_t$-free, and since $|S|=\mu=\mu(\max\{2s+1,t\})$, we can apply Theorem~\ref{connectifier} to $G_1$ and $S$. Note that every vertex in $S$ has a unique neighbor in $G_1$, and so no path in $G_1$ contains $\max\{2s+1,t\}\geq 3$ vertices from $S$. Consequently, 
  there is an induced subgraph $H_1$ of $G_1$ with $|H_1\cap S|=2s+1$ for which one of the following holds.
 \begin{itemize}
  \item $H_1$ is either a caterpillar or the line graph of a caterpillar with $H_1\cap S=\mathcal{Z}(H_1)$.
  \item $H_1$ is a subdivided star with root $r_1$ such that $\mathcal{Z}(H_1)\subseteq H_1\cap S\subseteq \mathcal{Z}(H_1)\cup \{r_1\}$.
\end{itemize}
  If $H_1$ is a caterpillar, then $G[H_1\cup \{a\}]$ contains
  a theta with ends $a$ and $a'$ for every vertex $a'\in H_1$ of degree more than two, a contradiction. Also, if the second bullet above holds, then since every vertex in $S$ is of degree one in $G_1$, we have $H_1\cap S= \mathcal{Z}(H_1)$, and so $r_1$ is not adjacent to $a$. But then $G[H_1\cup \{a\}]$ contains
  a theta with ends $x$ and $r_1$, a contradiction.
  It follows that $H_1$ is the line graph of a caterpillar with $|H_1\cap S|=2s+1$ and $H_1\cap S=\mathcal{Z}(H_1)$. This, together with the fact that every vertex in $H_1\cap S\subseteq S$ has a unique neighbor in $H_1\subseteq G$, implies that $H_1$ contains the line graph $H_2$ of a $1$-subdivision of a caterpillar with $|H_2\cap S|=s$ and $H_2\cap S=\mathcal{Z}(H_2)$. Let $S_2=H_2 \cap S=\mathcal{Z}(H_2)$; then $S_2$ is the set of all vertices of degree one in $H_2$, and we may assume without loss of generality that $S_2=\{a_1, \dots, a_{s}\}$.
  Let $G_2=G[H_2 \cup (\bigcup_{j=1}^{s}P_j)]$. It follows that $G_2\in \mathcal{C}_t$, $N_{G_2}(a)=S_2=\mathcal{Z}(H_2)$ and $a$ is trapped in $H_2\cup \{a\}$. Therefore, $H_2$ is an $a$-seed in $G_2$. Since $b\in G_2\setminus N_{G_2}[a]$, applying Theorem~\ref{noseed} to $G_2$ and $a$, we deduce that
  there exists $S_b \subseteq G_2\setminus \{a,b\}$
  such that $|S_b|< s$ and $S_b$ separates $a$ and $b$ in $G_2$. But
  $P_1, \dots, P_{s}$ are $s$ pairwise internally disjoint paths in $G_2$ from $a$ to $b$, a contradiction with Theorem~\ref{Menger}. This proves \eqref{nostableset}.\medskip

 \sloppy By \eqref{nostableset}, Theorem~\ref{classicalramsey}, and the definition of $\gamma$, $D^-$ contains a clique of cardinality $R_{tourn}(\nu+1)$.  This, along with the definition of $R_{tourn}(\cdot)$, implies that
  $D$ contains (as a subdigraph) a transitive tournament $K$ on $\nu+1$ vertices.
  We may assume without loss of generality that $V(K)=\{a_1, \dots, a_{\nu+1}\}$ such that for distinct $i,j\in [\nu+1]$, $(a_i,a_j)$ is an arc in $K$ if $i<j$. From the definition of $D$, it follows that $\{a_2,\ldots, a_{\nu+1},b\}$ is a stable set in $G$, and for all $i,j\in \{2,\ldots, \nu+1\}$ with $i<j$, $a_i$ has a neighbor in $P_j^*\setminus \{a_j\}$. Hence, $I=\{2,\ldots, \nu+1\}$ satisfies Theorem~\ref{banana}. This completes the proof.
\end{proof}

For positive integers $d$ and $r$, let $T_d^r$ denote the rooted tree in which every leaf is at distance $r$ from the root, the root has degree $d$, and every vertex that is neither a leaf nor the root has degree $d+1$. We need a result from \cite{KP}:

\begin{theorem}[Kierstead and Penrice \cite{KP}]\label{inducedtree}
    For all integers $d,r,s,t\geq 1$, there exists an integer $f=f(d,r,s,t)\geq 1$ such that if $G$ contains $T_f^f$ as a subgraph, then $G$ contains
    one of $K_{s,s}$, $K_t$ and $T_d^r$ as an induced subgraph.
\end{theorem}

The following lemma is the penultimate step in the proof of Theorem~\ref{mainthm}.

\begin{lemma}\label{gettreelemma}
    For all integers $d,r,t\geq 1$, there exists an integer $m=m(d,r,t)$ with the following property. Let $G\in \mathcal{C}_t$ be a graph, let $a,b\in V(G)$ be non-adjacent, and let $\{P_i:i\in [m]\}$ be a collection of $m$ pairwise internally disjoint paths in $G$ from $a$ to $b$. Then $G[\bigcup_{j=1}^{m}P_{j}]$ contains a subgraph $J$ isomorphic to $T_d^r$ such that $a\in J$, $a$ has degree $d$ in $J$ (that is, $a$ is the root of $J$), and $b\notin J$.
\end{lemma}
\begin{proof}
Let $d,t\geq 1$ be fixed. Let $m_1=d$. For every integer $r>1$, let 
$m_r=\psi(t,(m_{r-1}+1)d)$ where $\psi(\cdot,\cdot)$ is as in Theorem~\ref{banana}. We prove by induction on $r\geq 1$ that $m(d,r,t)=m_r$ satisfies Lemma~\ref{gettreelemma}. Let $P_1,\ldots, P_{m_r}$ be $m_r$ pairwise internally disjoint paths in $G$ from $a$ to $b$. Since $a$ and $b$ are not adjacent, it follows that for each $i\in [m_r]$, we have $P_i^*\neq \emptyset$. Let $a_{i}$ be the neighbor of $a$ in $P_i$. In particular, we have $b\notin \{a_i:i\in [m_r]\}$. Suppose first that $r=1$. Then we have $|\{a_i:i\in [m_1]\}|=m_1=d$, and so $G[\{a_i:i\in [m_r]\}\cup \{a\}]$ contains a (spanning) subgraph $J$ isomorphic to $T_d^1$ such that $a\in J$ and $a$ has degree $d$ in $J$, and we have $b\notin J$, as desired. Therefore, we may assume that $r\geq 2$. Since $m_r=\psi(t,(m_{r-1}+1)d)$, we can apply Theorem~\ref{banana} to $a,b$, and $\mathcal{P}=\{P_i:i\in [m_r]\}$. Without loss of generality, we may deduce that $\{a_1, \cdots, a_{(m_{r-1}+1)d},b\}$ is a stable set in $G$, and for all $i,j\in [(m_{r-1}+1)d]$ with $i<j$, $a_i$ has a neighbor in $P_j^*\setminus \{a_j\}$. For every $i\in [d]$, let $a'_i=a_{(i-1)m_{r-1}+i}$ and let
  $$A_i=\{(i-1)m_{r-1}+i+1, \ldots, (i-1)m_{r-1}+i+m_{r-1}\}.$$
  In particular, we have $|A_i|=m_{r-1}$. Then for each $i\in [d]$ and each $j\in A_i$, $a'_i$ has a neighbor in $P_j^*\setminus \{a_j\}$, and so there exists a path $Q_j$ in $G$ from $a'_i$ to $b$ with $Q_j^*\subseteq P_j^*$. Now, for every $i\in [d]$, $a_i'$ and $b$ are non-adjacent, and $\{Q_j:j\in A_i\}$ is a collection of $m_{r-1}$ pairwise internally disjoint paths in $G$ from $a_i'$ to $b$. It follows from the induction hypothesis that $G[\bigcup_{j\in A_i}Q_{j}]$ contains a subgraph $J_i$ isomorphic to $T_d^{r-1}$ such that $a_i'\in J_i$, $a_i'$ has degree $d$ in $J_i$, and $b\notin J_i$. But now $G[(\bigcup_{i=1}^{d}V(J_{i}))\cup \{a\}]\subseteq G[\bigcup_{j=1}^{m_r}P_{j}]$ contains a (spanning) subgraph $J$ isomorphic to $T_d^{r}$ such that $a\in J$, $a$ has degree $d$ in $J$, and $b\notin J$. This completes the proof of Lemma~\ref{gettreelemma}.
\end{proof}

Finally, we prove Theorem~\ref{mainthm}, which we restate:

\setcounter{section}{1}
\setcounter{theorem}{7}

\begin{theorem}
  For every tree $F$ and every integer $t \geq 1$, there exists an integer $\tau(F,t)\geq 1$ such that
  every graph in $\mathcal{C}_t(F)$ has treewidth at most $\tau(F,t)$.
\end{theorem}
\begin{proof}

Let $d$ and $r$ be the maximum degree and the radius of $F$, respectively. It follows that $T_d^r$ contains $F$ as an induced subgraph. Let $f=f(d,r,3,t)$ be as in Theorem~\ref{inducedtree} and let $m=m(f,f,t)$ be as in Lemma~\ref{gettreelemma}. Let $\beta(\cdot,\cdot)$ be as in Corollary~\ref{noblocksmalltw_Ct}. We claim that $\tau(F,t)=\beta(\max\{m,t+1\},t)$ satisfies Theorem~\ref{mainthm}. Suppose for a contradiction that $\tw(G)>\tau$ for some $G\in \mathcal{C}_t(F)$. By Corollary~\ref{noblocksmalltw_Ct}, $G$ contains a strong $\max\{m,t+1\}$-block $B$. Consequently, since $G$ is $K_t$-free, there are two distinct and non-adjacent vertices $a,b\in B$, and $m$ pairwise internally disjoint paths $P_1,\ldots, P_{m}$ in $G$ from $a$ to $b$. It follows from Lemma~\ref{gettreelemma} that $G$ contains $T_f^f$ as a subgraph. Also, since  $G \in \mathcal{C}_t(F)\subseteq \mathcal{C}_t$, $G$ is $(K_{3,3}, K_t)$-free. But now by Theorem~\ref{inducedtree}, $G$ contains $T_d^r$, and so $F$, as an induced subgraph, a contradiction. This completes the proof.
\end{proof}


\begin{thebibliography} {99}

  \bibitem{aboulker}
 P. Aboulker, I. Adler, E. J. Kim, N. L. D. Sintiari, and N. Trotignon.
\newblock {``On the treewidth of even-hole-free graphs.''}
{\em European Journal of Combinatorics} {\bf 98}, (2021), 103394. 

    
  \bibitem{twvii} T. Abrishami, B. Alecu, M. Chudnovsky, S. Hajebi, and S. Spirkl,
  ``Induced subgraphs and tree decompositions VII. Basic obstructions in $H$-free graphs.'' \emph{arXiv:2212.02737}, (2022).
  
  \bibitem{wallpaper}
  T. Abrishami, M. Chudnovsky, C. Dibek, S. Hajebi, P. Rz\k{a}\.{z}ewski, S. Spirkl, and K. Vu\v{s}kovi\'c, ``Induced subgraphs and tree decompositions II. Toward walls and their line graphs in graphs of bounded degree.''
  {\em arXiv:2108.01162}, (2021). 
  
\bibitem{TWI}
  T. Abrishami, M. Chudnovsky  and K. Vu\v{s}kovi\'c, ``Induced subgraphs and tree decompositions I. Even-hole-free graphs of bounded degree.''
   \textit{J. Combin. Theory Ser. B}, {\bf 157}  (2022), 144-175.

   \bibitem{ajtai}
 M. Ajtai, J. Koml\'{o}s and E. Szemer\'{e}di. \newblock {``A note on Ramsey numbers.''}
{\em  J. Combinatorial Theory, Ser. A} {\bf 29} (1980), 354–360.
  
  
\bibitem{Bodlaender1988DynamicTreewidth}
 H.~L. Bodlaender.
\newblock {``Dynamic programming on graphs with bounded treewidth.''}
\newblock Springer, Berlin, Heidelberg, (1988), pp.~105--118.

\bibitem{evenholetrianglefree} K. Cameron, M.V. da Silva, S. Huang, and K. Vu\v{s}kovi\'c,
``Structure and algorithms for (cap, even hole)-free
graphs.'' 
\emph{Discrete Mathematics} {\bf 341}, 2 (2018), 463-473.


\bibitem{SPGT} M. Chudnovsky, N. Robertson, P. Seymour, and  R. Thomas,
``The strong perfect graph theorem.'' 
\emph{Annals of Math} {\bf 164} (2006), 51-229.

\bibitem{bisimp2}
  M. Chudnovsky and P. Seymour, ``Even-hole-free graphs still have bisimplicial vertices.'' \emph{arXiv:1909.10967}, (2019).
  
\bibitem{threeinatree}
  M. Chudnovsky and P. Seymour, ``The three-in-a-tree problem.''
\emph{Combinatorica} \textbf{30}, 4 (2010): 387-417.




\bibitem{Davies}
 J. Davies, ``Vertex-minor-closed classes are $\chi$-bounded.'' 
\emph{arXiv:2008.05069}, (2020).


\bibitem{Davies2}
 J. Davies, appeared in an Oberwolfach technical report {\em DOI:10.4171/OWR/2022/1}.

 
\bibitem{tighttw}
 J. Erde and D. Wei\ss auer. \newblock {``A short derivation of the structure theorem for graphs with excluded topological minors.''}
{\em  SIAM Journal of Discrete Mathematics} {\bf 33}, 3 (2019), 1654--1661.

\bibitem{Grohe}
M. Grohe and D. Marx. ``Structure theorem and isomorphism test for graphs with excluded topological subgraphs,'' {\em SIAM Journal on Computing} \textbf{44}, 1 (2015), 114--159.

 
  \bibitem{KP} H.A. Kierstead and S. G. Penrice, ``Radius two trees specify $\chi$-bounded classes.'' \emph{J. Graph Theory} \textbf{18}, 2 (1994): 119--129.

\bibitem{Korhonen}
  T. Korhonen, ``Grid Induced Minor Theorem for Graphs of Small degree.''
 {\em arXiv:2203.13233}, (2022).

    
    \bibitem{lozin} V. Lozin and I. Razgon. ``Tree-width dichotomy.'' \emph{European J. Combinatorics} \textbf{103} (2022): 103517.
    


    \bibitem{Menger}   K. Menger, ``Zur allgemeinen Kurventheorie.'' \textit {Fund. Math.}
      {\bf 10}, 1927, 96--115.

      
  \bibitem{RS-GMV}
N. Robertson and P. Seymour. ``Graph minors. V. Excluding a planar graph.'' \textit{J. Combin. Theory Ser. B}, \textbf{41} (1) (1996), 92–114.
  

  \bibitem{layered-wheels} N.L.D. Sintiari and N. Trotignon.
  ``(Theta, triangle)-free and (even-hole, $K_4$)-free graphs. Part 1: Layered wheels.'' {\em  J. Graph Theory} \textbf{97} (4) (2021), 475-509. 



\bibitem{Trotignon} N. Trotignon, {\em private communication}, 2021.
      
  \end{thebibliography}
\end{document}